\documentclass[final]{siamart1116}

\usepackage{float}
\usepackage{lipsum}
\usepackage{amsfonts}
\usepackage{graphicx}
\usepackage{epstopdf}
\usepackage{algorithmic}
\ifpdf
  \DeclareGraphicsExtensions{.eps,.pdf,.png,.jpg}
\else
  \DeclareGraphicsExtensions{.eps}
\fi

\numberwithin{theorem}{section}

\newcommand{\TheTitle}{Stochastic Primal-Dual Hybrid Gradient Algorithm with Arbitrary Sampling and Imaging Applications}
\newcommand{\TheAuthors}{A. Chambolle, M. J. Ehrhardt, P. Richt\'{a}rik and C.-B. Sch\"{o}nlieb}

\headers{Stochastic Primal-Dual Hybrid Gradient Algorithm}{\TheAuthors}

\title{{\TheTitle}\thanks{Submitted to the editors 15/06/2017.
\funding{A. C. benefited from a support of the ANR, \enquote{EANOI} Project I1148 / ANR-12-IS01-0003 (joint with FWF). Part of this work was done while he was hosted in Churchill College and DAMTP, Centre for Mathematical Sciences, University of Cambridge, thanks to a support of the French Embassy in the UK and the Cantab Capital Institute for Mathematics of Information. M. J. E. and C.-B. S. acknowledge support from Leverhulme Trust project \enquote{Breaking the non-convexity barrier}, EPSRC grant \enquote{EP/M00483X/1}, EPSRC centre \enquote{EP/N014588/1}, the Cantab Capital Institute for the Mathematics of Information, and from CHiPS (Horizon 2020 RISE project grant). M. J. E. has carried out initial work supported by the EPSRC platform grant \enquote{EP/M020533/1}. Moreover, C.-B. S. is thankful for support by the Alan Turing Institute. P. R. acknowledges the support of EPSRC Fellowship in Mathematical Sciences \enquote{EP/N005538/1} entitled \enquote{Randomized algorithms for extreme convex optimization}.}}}

\author{
  Antonin Chambolle\thanks{CMAP, Ecole Polytechnique, CNRS, France (\email{antonin.chambolle@cmap.polytechnique.fr}).}%
  \and%
  Matthias J. Ehrhardt\thanks{Department for Applied Mathematics and Theoretical Physics, University of Cambridge, Cambridge CB3 0WA, UK (\email{m.j.ehrhardt@damtp.cam.ac.uk}, \email{cbs31@cam.ac.uk}).}%
  \and%
  Peter Richt\'{a}rik\thanks{(i) Visual Computing Center \& Extreme Computing Research Center, KAUST, Thuwal, Saudi Arabia; (ii) School of Mathematics, University of Edinburgh, Edinburgh, United Kingdom; (iii) The Alan Turing Institute, London, United Kingdom (\email{peter.richtarik@kaust.edu.sa, peter.richtarik@ed.ac.uk}).}%
  \and%
  Carola-Bibiane Sch\"{o}nlieb\footnotemark[3]%
}

\usepackage{amsopn}%
\usepackage[utf8]{inputenc}
\usepackage[english]{babel}

\usepackage{csquotes}
\usepackage{xcolor}
\usepackage{amsmath,amssymb}

\newcommand{\E}{\mathbb{E}}
\newcommand{\sRI}{\sR_\infty}

\newcommand{\Angle}[2]{\langle #1, #2 \rangle}
\newcommand{\Prox}[3]{\operatorname{prox}^{#1}_{#2}\left(#3\right)}

\def\ParamTheta{\theta}

\def\ParamSigmaMat{\mS}

\def\ParamTauMat{\mT}
\def\opt{\sharp}
\def\iter{k}
\def\Iter{K}

\def\E{\mathbb{E}}

\usepackage{mathtools}
\newcommand{\De}{\coloneqq}

\newtheorem{remark}{Remark}
\newtheorem{example}{Example}

\usepackage{hyperref}

\newcommand{\mA}{{\bf A}}
\newcommand{\mB}{{\bf B}}
\newcommand{\mC}{{\bf C}}
\newcommand{\mD}{{\bf D}}

\newcommand{\mI}{{\bf I}}

\newcommand{\mM}{{\bf M}}

\newcommand{\mQ}{{\bf Q}}

\newcommand{\mS}{{\bf S}}
\newcommand{\mT}{{\bf T}}

\newcommand{\mX}{{\bf X}}
\newcommand{\mY}{{\bf Y}}

\newcommand{\sB}{{\mathbb B}}

\newcommand{\sN}{{\mathbb N}}

\newcommand{\sP}{{\mathbb P}}

\newcommand{\sR}{{\mathbb R}}
\newcommand{\sS}{{\mathbb S}}

\newcommand{\sW}{{\mathbb W}}
\newcommand{\sX}{{\mathbb X}}
\newcommand{\sY}{{\mathbb Y}}

\newcommand{\cF}{{\mathcal F}}
\newcommand{\cG}{{\mathcal G}}
\newcommand{\cH}{{\mathcal H}}

\newcommand{\wo}{w^\opt}

\newcommand{\wk}{w^{(\iter)}}
\newcommand{\wkp}{w^{(\iter+1)}}

\newcommand{\xo}{x^\opt}
\newcommand{\xn}{x^{(0)}}
\newcommand{\xk}{x^{(\iter)}}
\newcommand{\xkp}{x^{(\iter+1)}}

\newcommand{\xK}{x^{(\Iter)}}

\newcommand{\yo}{y^\opt}
\newcommand{\yn}{y^{(0)}}
\newcommand{\yk}{y^{(\iter)}}
\newcommand{\ykp}{y^{(\iter+1)}}
\newcommand{\ykm}{y^{(\iter-1)}}
\newcommand{\yK}{y^{(\Iter)}}
\newcommand{\yKm}{y^{(\Iter-1)}}

\newcommand{\yhkp}{\hat y^{(\iter+1)}}

\newcommand{\ybn}{\overline y^{(0)}}
\newcommand{\ybk}{\overline y^{(\iter)}}
\newcommand{\ybkp}{\overline y^{(\iter+1)}}

\newcommand{\tn}{\tau_{(0)}}
\newcommand{\tk}{\tau_{(\iter)}}
\newcommand{\tkp}{\tau_{(\iter+1)}}

\newcommand{\tK}{\tau_{(\Iter)}}

\newcommand{\thk}{\theta_{(\iter)}}

\newcommand{\thkm}{\theta_{(\iter-1)}}

\newcommand{\sn}{\sigma^{(0)}}
\newcommand{\sk}{\sigma^{(\iter)}}
\newcommand{\skp}{\sigma^{(\iter+1)}}

\newcommand{\stn}{\tilde \sigma_{(0)}}
\newcommand{\stk}{\tilde \sigma_{(\iter)}}
\newcommand{\stkp}{\tilde \sigma_{(\iter+1)}}

\newcommand{\stK}{\tilde \sigma_{(\Iter)}}

\newcommand{\Ek}{\E^{(\iter)}}
\newcommand{\Ekp}{\E^{(\iter+1)}}
\newcommand{\Ekm}{\E^{(\iter-1)}}

\newcommand{\Dn}{\Delta^{(0)}}
\newcommand{\Dk}{\Delta^{(\iter)}}
\newcommand{\Dkp}{\Delta^{(\iter+1)}}
\newcommand{\DK}{\Delta^{(\Iter)}}

\newcommand{\mAa}{\mA^\ast}

\newcommand{\mSn}{\mS_{(0)}}
\newcommand{\mSk}{\mS_{(\iter)}}
\newcommand{\mSkp}{\mS_{(\iter+1)}}

\newcommand{\mTk}{\mT_{(\iter)}}

\newcommand{\mYn}{\mY_{(0)}}
\newcommand{\mYk}{\mY_{(\iter)}}
\newcommand{\mYkp}{\mY_{(\iter+1)}}
\newcommand{\mYK}{\mY_{(\Iter)}}

\newcommand{\sSk}{\sS^{(\iter)}}
\newcommand{\sSkp}{\sS^{(\iter+1)}}

\usepackage[colorinlistoftodos,bordercolor=orange,backgroundcolor=orange!20,linecolor=orange,textsize=scriptsize]{todonotes}

\newcommand{\new}[1]{#1}

\ifpdf
\hypersetup{
  pdftitle={\TheTitle},
  pdfauthor={\TheAuthors}
}
\fi

\begin{document}

\maketitle

\begin{abstract}
We propose a stochastic extension of the primal-dual hybrid gradient algorithm studied by Chambolle and Pock in 2011 to solve saddle point problems that are separable in the dual variable. The analysis is carried out for general convex-concave saddle point problems and problems that are either partially smooth / strongly convex or fully smooth / strongly convex. We perform the analysis for arbitrary samplings of dual variables, and obtain known deterministic results as a special case.  Several  variants of our stochastic method significantly outperform the deterministic variant on a variety of imaging tasks.
\end{abstract}



\section{Introduction}
Many modern problems in a variety of disciplines (imaging, machine learning, statistics, etc.) can be formulated as convex optimization problems. Instead of solving the optimization problems directly, it is often advantageous to reformulate the problem as a saddle point problem. A very popular algorithm to solve such saddle point problems is the primal-dual hybrid gradient (PDHG)\footnote{We follow the terminology of \cite{Chambolle2016} and call the algorithm simply PDHG. It corresponds to PDHGMu and PDHGMp in \cite{Esser2010}.} algorithm \cite{Pock2009,Esser2010,Chambolle2011,Pock2011,Chambolle2016,Chambolle2016a}. It has been used to solve a vast amount of state-of-the-art problems---to name a few examples in imaging: image denoising with the structure tensor \cite{Estellers2015}, total generalized variation denoising \cite{Bredies2015}, dynamic regularization \cite{Benning2016}, multi-modal medical imaging \cite{Knoll2016}, multi-spectral medical imaging \cite{Rigie2015}, computation of non-linear eigenfunctions \cite{Gilboa2015}, regularization with directional total generalized variation \cite{Kongskov2017}. Its popularity stems from two facts: First, it is very simple and therefore easy to implement. Second, it involves only simple operations like matrix-vector multiplications and evaluations of proximal operators which are for many problems of interest simple and in closed-form or easy to compute iteratively, cf. e.g. \cite{Parikh2014}. However, for large problems that are encountered in many real world applications, even these simple operations might be still too costly to perform too often.

We propose a stochastic extension of the PDHG for saddle point problems that are separable in the dual variable (cf. e.g. \cite{Dang2014, Zhang2015, Zhu2015, Peng2016}) where not all but only a few of these operations are performed in every iteration. Moreover, as in incremental optimization algorithms \cite{Tseng1998,Nedic2001,Blatt2007,Bertsekas2011,Bertsekas2011a,Schmidt2016,Oliviera2016} over the course of the iterations we continuously build up information from previous iterations which reduces variance and thereby negative effects of stochasticity. Non-uniform samplings \cite{Richtarik2014, Qu2016, Zhang2015, Qu2015, Allen-Zhu2016} have been proven very efficient for stochastic optimization. In this work we use the expected separable overapproximation framework of \cite{Qu2016, Qu2015, Richtarik2016} to prove all statements for all non-trivial and iteration-independent samplings.

\paragraph{Related Work} The proposed algorithm can be seen as a generalization of the algorithm of \cite{Dang2014, Zhu2015, Zhang2015} to arbitrary blocks and a much wider class of samplings. Moreover, in contrast to their results, our results generalize the deterministic case considered in \cite{Pock2009, Chambolle2011, Pock2011, Chambolle2016a}. Fercoq and Bianchi \cite{Fercoq2015} proposed a stochastic primal-dual algorithm with explicit gradient steps that allows for larger step sizes by averaging over previous iterates, however, this comes at the cost of prohibitively large memory requirements. Similar memory issues are encountered by a primal-dual algorithm of \cite{Balamurugan2016a}. It is related to forward-backward splitting \cite{Lions1979} and averaged gradient descent \cite{Blatt2007, Defazio2014} and therefore suffers the same memory issues as the averaged gradient descent. Moreover, Valkonen proposed a stochastic primal-dual algorithm that can exploit partial strong convexity of the saddle point functional \cite{Valkonen2016}. Randomized versions of the alternating direction method of multipliers are discussed for instance in \cite{Zhong2014, Gao2016}. In contrast to other works on stochastic primal-dual algorithms \cite{Pesquet2015, Wen2016}, our analysis is not based on Fej\'er monotonicity \cite{Combettes2015}. We therefore do not prove almost sure convergence of the sequence but prove a variety of convergence rates depending on strong convexity assumptions instead.

As a word of warning, our contribution should not be mistaken by other \enquote{stochastic} primal-dual algorithms, where errors in the computation of matrix-vector products and evaluation of proximal operators are modeled by random variables, cf. e.g. \cite{Pesquet2015, Combettes2015, Rosasco2015}. In our work we deliberately choose to compute only a subset of a whole iteration to save computational cost. These two notations are related but are certainly not the same.

\subsection{Contributions}
We briefly mention the main contributions of our work.

\paragraph{Generalization of Deterministic Case} The proposed stochastic algorithm is a direct generalization of the deterministic setting \cite{Pock2009, Chambolle2011, Pock2011, Chambolle2016, Chambolle2016a}. In the degenerate case where in every iteration all computations are performed, our algorithm coincides with the original deterministic algorithm. Moreover, the same holds true for our analysis of the stochastic algorithm where we recover almost all deterministic statements \cite{Chambolle2011, Pock2011} in this degenerate case. Therefore, the theorems for both the deterministic and the stochastic case can be combined by a single proof.

\paragraph{Better Rates} Our analysis extends the simple setting of \cite{Zhang2015} such that the strong convexity assumptions and the sampling do not have to be uniform. Even in the special case of uniform strong convexity and uniform sampling, the proven convergence rates are slightly better than the ones proven in \cite{Zhang2015}.

\paragraph{Arbitrary Sampling} The proposed algorithm is guaranteed to converge under a very general class of samplings \cite{Qu2016, Qu2015, Richtarik2016} and thereby generalizes also the algorithm of \cite{Zhang2015} which has only been analyzed for two specific samplings. As long as the sampling is independent and identically distributed over the iterations and all computations have non-zero probability to be carried out, the theory holds and the algorithm will converge with the proven convergence rates.

\paragraph{Acceleration} We propose an acceleration of the stochastic primal-dual algorithm which accelerates the convergence from $\mathcal O(1/K)$ to $\mathcal O(1/K^2)$ if parts of the saddle point functional are strongly convex and thereby results in a significantly faster algorithm.

\paragraph{Scaling Invariance} In the strongly convex case, we propose parameters for several serial samplings (uniform, importance, optimal), all based on the condition numbers of the problem and thereby independent of scaling.

\section{General Problem}
Let $\sX, \sY_i, i = 1, \ldots, n$ be real Hilbert spaces of any dimension and define the product space $\sY \De \prod_{i=1}^n \sY_i$. For $y\in \sY$, we shall write $y=(y_1,y_2,\dots,y_n)$, where $y_i \in \sY_i$. Further, we consider the natural inner product on the product space $\sY$ given by $\langle y, z \rangle = \sum_{i=1}^n \langle y_i, z_i \rangle$, where $y_i, z_i \in \sY_i$. This inner product induces the norm $\|y\|^2 = \sum_{i=1}^n \|y_i\|^2$. Moreover, for simplicity we will consider the space $\sW \De \sX \times \sY$ that combines both primal and dual variables.

Let $\mA : \sX \to \sY$ be a bounded linear operator. Due to the product space nature of $\sY$, we have $(\mA x)_i = \mA_i x$, where  $\mA_i : \sX \to \sY_i$ are linear operators. The adjoint of $\mA$ is given by $\mAa y = \sum_{i=1}^n \mAa_i y_i$. Moreover, let $f : \sY \to \sRI \De \sR \cup \{+\infty\}$ and $g : \sX \to \sRI$ be convex functions. In particular, we assume that $f$ is separable, i.e. $f(y) = \sum_{i=1}^n f_i(y_i)$.

Given the setup described above, we consider the optimization problem
\begin{align}
    \min_{x \in \sX} \left\{ \Phi(x) := \sum_{i=1}^n f_i(\mA_i x) + g(x) \right\} \, . \label{EQU:MINPROB}
\end{align}
Instead of solving \cref{EQU:MINPROB} directly, it is often desirable to reformulate the problem as a saddle point problem with the help of the Fenchel conjugate. If $f$ is proper, convex, and lower semi-continuous, then $f(y) = f^{\ast\ast}(y) = \sup_{z \in \sY} \langle z, y \rangle - f^\ast(z)$ where $f^\ast : \sY \to \sR \cup \{-\infty, +\infty\}$, $f^\ast(y) = \sum_{i=1}^n f_i^\ast(y_i)$ is the Fenchel conjugate of $f$ (and $f^{\ast\ast}$ its biconjugate defined as the conjugate of the conjugate). Then solving \cref{EQU:MINPROB} is equivalent to finding the primal part $x$ of a solution to the saddle point problem (called a saddle point)
\begin{align}
    \underset{x \in \sX}{\operatorname{min\phantom{p}\!\!\!}}
    \sup_{y \in \sY} \left\{ \Psi(x, y) \De \sum_{i=1}^n \Angle{\mA_i x}{y_i} - f_i^\ast(y_i) + g(x) \right\} \, . \label{EQU:SADDLE}
\end{align}
We will assume that the saddle point problem \cref{EQU:SADDLE} has a solution. For conditions for existence and uniqueness, we refer the reader to \cite{Bauschke2011}. A saddle point $\wo = (\xo, \yo) = (\xo, \yo_1, \ldots, \yo_n) \in \sW$ satisfies the optimality conditions
\begin{align*}
    \mA_i \xo \in \partial f^\ast_i(\yo_i) \quad i = 1, \ldots, n, \qquad
    -\mAa \yo \in \partial g(\xo) \, .
\end{align*}
An important notion in this work is \emph{strong convexity}. A functional $g$ is called $\mu_g$-convex if $g - \frac{\mu_g}2 \|\cdot\|^2$ is convex. In general, we assume that $g$ is $\mu_g$-convex, $f_i^\ast$ is $\mu_i$-convex with nonnegative strong convexity parameters $\mu_g, \mu_i \geq 0$. The convergence results in this contribution cover three different cases of regularity: i) no strong convexity $\mu_g, \mu_i = 0$, ii) semi strong convexity $\mu_g > 0$ or $\mu_i > 0$ and iii) full strong convexity $\mu_g, \mu_i > 0$. For notational convenience we make use of the operator $\mM \De \operatorname{diag}(\mu_1 \mI, \ldots, \mu_n \mI)$.

A very popular algorithm to solve the saddle point problem \cref{EQU:SADDLE} is the Primal-Dual Hybrid Gradient (PDHG) algorithm \cite{Pock2009,Esser2010, Chambolle2011, Pock2011, Chambolle2016, Chambolle2016a}. It reads (with extrapolation on $y$)
\begin{align*}
    \xkp &= \Prox{\tau}{g}{\xk - \tau \mAa \ybk} \\
    \ykp &= \Prox{\sigma}{f^\ast}{\yk + \sigma \mA \xkp} \\
    \ybkp &= \ykp + \ParamTheta\left(\ykp - \yk\right) \, ,
\end{align*}
where the \emph{proximal operator (or proximity / resolvent operator)} is defined as
\begin{align*}
    \Prox{\tau}{f}{y} \De \arg \min_{x \in \sX} \left\{ \frac12 \| x - y \|^2_{\tau^{-1}} + f(x) \right\}
\end{align*}
and the weighted norm by $\|x\|_{\tau^{-1}}^2 = \Angle{\tau^{-1} x}{x}$. Its convergence is guaranteed if the step size parameters $\sigma, \tau$ are positive and satisfy $\sigma \tau \|\mA\|^2 < 1, \theta = 1$ \cite{Chambolle2011}. Note that the definition of the proximal operator is well-defined for an \emph{operator-valued} step size $\tau$. In the case of a separable function $f$ and with operator-valued step sizes the PDHG algorithm takes the form
\begin{subequations}
\begin{align}
    \xkp &= \Prox{\mT}{g}{\xk - \mT \mAa \ybk} \label{EQU:PDHG:SEP:PRIMAL} \\
    \ykp_i &= \Prox{\ParamSigmaMat_i}{f^\ast_i}{\yk_i + \ParamSigmaMat_i \mA_i \xkp} \qquad i = 1, \ldots, n \label{EQU:PDHG:SEP:DUAL} \\
    \overline y^{(\iter+1)} &= \ykp + \ParamTheta\left(\ykp - \yk\right) \, .
\end{align}
\end{subequations}
Here the step size parameters $\ParamSigmaMat = \operatorname{diag}(\ParamSigmaMat_1, \ldots, \ParamSigmaMat_n)$ (a block diagonal operator), $\ParamSigmaMat_1, \ldots, \ParamSigmaMat_n$ and $\ParamTauMat$ are symmetric and positive definite. The algorithm is guaranteed to converge if $\|\ParamSigmaMat^{1/2} \mA \ParamTauMat^{1/2}\| < 1$ and $\ParamTheta = 1$ \cite{Pock2011}.

\section{Algorithm}
In this work we extend the PDHG algorithm to a stochastic setting where in each iteration we update a random subset $\sS$ of the dual variables \cref{EQU:PDHG:SEP:DUAL}. This subset is sampled in an i.i.d. fashion from a fixed but otherwise {\em arbitrary} distribution, whence the name \enquote{arbitrary sampling}. In order to guarantee convergence, it is necessary to assume that the sampling is \enquote{proper} \cite{PCDM, Qu2015}. A sampling is proper if  for each dual variable $i$ we have $i \in \sS$ with a positive probability $p_i>0$. Examples of proper samplings include the \emph{full sampling} where $\sS = \{1, \ldots, n\}$ with probability 1 and serial sampling where $\sS = \{i\}$ is chosen with probability $p_i$. It is important to note that also other samplings are admissible. For instance for $n = 3$, consider the sampling that selects $\sS = \{1,2\}$ with probability $1/3$ and $\mathbb S=\{2,3\}$ with probability $2/3$. Then the probabilities for the three blocks are $p_1 = 1/3$, $p_2 = 1$ and $p_3 = 2/3$ which makes it a proper sampling. However, if only $\sS = \{1,2\}$ is chosen with probability 1, then this sampling is not proper as the probability for the third block is zero: $p_3 = 0$.

\begin{algorithm}
\caption{Stochastic Primal-Dual Hybrid Gradient algorithm (\textbf{SPDHG}). \newline \textbf{Input:} $\xn, \yn$, $\mS = \operatorname{diag}(\mS_1, \ldots, \mS_n)$, $\mT, \theta$, $\sSk$, $\Iter$. \textbf{Initialize:} $\ybn = \yn$} \label{ALG:PDHG}
\begin{algorithmic}
    \FOR{$\iter = 0, \ldots, \Iter-1$}
    \STATE{$\xkp = \Prox{\mT}{g}{\xk - \mT \mAa \ybk}$}
    \STATE{Select $\sSkp \subset \{1, \ldots, n\}$.}
    \STATE{$\ykp_i = \begin{cases} \Prox{\mS_i}{f^\ast_i}{\yk_i + \mS_i \mA_i \xkp} & \text{if $i \in \sSkp$} \\ \yk_i & \text{else} \end{cases}$}
    \STATE{$\ybkp = \ykp + \theta \mQ \left(\ykp - \yk\right)$}
    \ENDFOR
\end{algorithmic}
\end{algorithm}

The algorithm we propose is formalized as \cref{ALG:PDHG}. As in the original PDHG, the step size parameters $\mT, \mS_i$ have to be self-adjoint and positive definite operators for the updates to be well-defined. The extrapolation is performed with a scalar $\theta > 0$ and an operator \new{$\mQ \De \operatorname{diag}(p_1^{-1} \mI, \ldots, p_n^{-1} \mI)$} of probabilities $p_i$ that an index is selected in each iteration.

\begin{remark}
Both, the primal and dual iterates $\xk$ and $\yk$ are random variables but only the dual iterate $\yk$ depends on the sampling $\sSk$. However, $\xk$ depends of course on all previous samplings $\sS^{(i)}, i < k$.
\end{remark}

\begin{remark}
Due to the sampling each iteration requires both $\mA_i$ and $\mAa_i$ to be evaluated only for each selected index $i \in \sSkp$. To see this, note that
\begin{align*}
\mAa \overline y^{(\iter+1)}
&= \mAa \yk + \sum_{i \in \sSkp} \left(1 + \frac{\theta}{p_i} \right) \mAa_i \left(\ykp_i - \yk_i\right)
\end{align*}
where $\mAa \yk$ can be stored from the previous iteration (needs the same memory as the primal variable $x$) and the operators $\mAa_i$ are evaluated only for $i \in \sSkp$.
\end{remark}

\section{General Convex Case}
We first analyze the convergence of \cref{ALG:PDHG} in the general convex case without making use of any strong convexity or smoothness assumptions. In order to analyze the convergence for the large class of samplings described in the previous section we make use of the \emph{expected separable overapproximation (ESO)} inequality  \cite{Qu2015}.
\begin{definition}[Expected Separable Overapproximation (ESO) \cite{Qu2015}]
Let $\sS \subset \{1, \ldots, n\}$ be a random set and $p_i \De \mathbb P(i \in \sS)$ the probability that $i$ is in $\sS$. We say that $\{v_i\} \subset \sR^n$ fulfill the \emph{ESO inequality} (depending on $\mC = [\mC_1; \ldots; \mC_n], \mC^\ast z = \sum_{i=1}^n C_i^\ast z_i$ and $\sS$) if for all $z \in \sY$ it holds that
\begin{align}
\E_\sS \left\|\sum_{i \in \sS} \mC_i^\ast z_i \right\|^2 \leq \sum_{i=1}^n p_i v_i \|z_i\|^2 \, . \label{EQU:ESO}
\end{align}
Such parameters $\{v_i\}$ are called \emph{ESO parameters}.
\end{definition}

\begin{remark}
Note that for any bounded linear operator $\mC$ such parameters always exist but are obviously not unique. For the efficiency of the algorithm it is desirable to find ESO parameters such that \cref{EQU:ESO} is as tight as possible; i.e., we want the parameters  $\{v_i\}$ be small. As we shall see, the ESO parameters influence the choice of the extrapolation parameter $\theta$ in the strongly convex case.
\end{remark}

The ESO inequality was first proposed by Richt\'{a}rik and Tak\'{a}\v{c}~\cite{PCDM} to study parallel coordinate descent methods in the context of {\em uniform} samplings, which are samplings for which $p_i = p_j$ for all $i, j$. Improved bounds for ESO parameters were obtained in \cite{APPROX} and used in the context of accelerated coordinate descent. Qu et al.~\cite{Qu2015} perform an in-depth study of ESO parameters. The ESO inequality is also critical in the study mini-batch stochastic gradient descent with \cite{mS2GD} or without \cite{MinibatchSGD} variance reduction.

We will frequently need to estimate the expected value of inner products which we will do by means of ESO parameters. Recall that we defined weighted norms as $\|x\|_{\mT^{-1}}^2 \De \langle \mT^{-1} x, x \rangle$. The proof of this lemma can be found in the appendix.
\begin{lemma} \label{LEM:ESO}
Let $\sS \subset \{1, \ldots, n\}$ be a random set and $y^+_i = \hat y_i$ if $i \in \sS$ and $y_i$ otherwise. Moreover, let $\{v_i\}$ be some ESO parameters of \new{$\mS^{1/2} \mA \mT^{1/2}$} and $p_i = \sP(i \in \sS)$. Then for any $x \in \sX$ and $c > 0$
\begin{align*}
  2 \E_{\sS} \Angle{\mQ \mA x}{y^+ - y} \geq - \E_{\sS} \left\{ \frac1c \|x\|_{\mT^{-1}}^2 + c \max_i \frac{v_i}{p_i} \|y^+ - y\|_{\mQ \mS^{-1}}^2 \right\} \, .
\end{align*}
\end{lemma}

\begin{example}[Full Sampling]
Let $\sS = \{1, \ldots, n\}$ with probability 1 such that $p_i = \mathbb P(i \in \sS) = 1$. Then $v_i = \|\mS^{1/2} \mA \mT^{1/2}\|^2$ are ESO parameters of $\mS^{1/2} \mA \mT^{1/2}$. Thus, the deterministic condition on convergence, $\|\mS^{1/2} \mA \mT^{1/2}\|^2 < 1$, implies a bound on some ESO parameters $v_i < p_i$.
\end{example}

\begin{example}[Serial Sampling] \label{EXA:SERIAL}
Let $\sS = \{i\}$ be chosen with probability $p_i > 0$. Then $v_i = \|\mS_i^{1/2} \mA_i \mT^{1/2}\|^2$ are ESO parameters of $\mS^{1/2} \mA \mT^{1/2}$.
\end{example}

\new{The analysis for the general convex case will use the notation of \emph{Bregman distance} which is defined for any function $f : \sX \to \sRI$, $x, y \in \sX$ and $q \in \partial f(y)$ in the subdifferential of $f$ at $y$ as
$$D_f^{q}(x, y) \De f(x) - f(y) - \langle q, x - y \rangle \, .$$
Next to Bregman distances, one can measure optimality by the partial primal-dual gap. Let $\sB_1 \times \sB_2 \subset \sW = \sX \times \sY$, then we define the partial primal-dual gap as
\begin{align*}
    G_{\sB_1 \times \sB_2}(x, y) \De \sup_{\tilde y \in \sB_2} \Psi(x, \tilde y) - \inf_{\tilde x \in \sB_1} \Psi(\tilde x, y) \, .
\end{align*}
It is convenient to define $\sB \De \sB_1 \times \sB_2 \subset \sW$ and to denote the gap as $G_{\sB}(w) \De G_{\sB_1 \times \sB_2}(x, y).$ Note that if $\sB$ contains a saddle point $\wo = (\xo, \yo)$, then we have that
\begin{align*}
    G_{\sB}(w)
    \geq \Psi(x, \yo) - \Psi(\xo, y)
    = D_g^{-\mAa y^\opt}(x, \xo) + D_{f^\ast}^{\mA \xo}(y, \yo) = D_h^q(w, \wo) \geq 0
\end{align*}
where the first equality is obtained by adding a zero and we used $h(w) \De g(x) + f^\ast(y)$ and $q \De (-\mAa y^\opt, \mA x^\opt) \in \partial h(w^\opt)$ for the last equality. The non-negativity stems from the fact that Bregman distances of convex functionals are non-negative and $h$ is convex indeed.

We will make frequent use of the following \enquote{distance functions}
$$\cF_i(y_i | \tilde x, \tilde y_i) \De f_i^\ast(y_i) - f_i^\ast(\tilde y_i) - \langle \mA_i \tilde x, y_i - \tilde y_i \rangle$$
and $\cF(y | \tilde w) \De \sum_{i=1}^n \mathcal F_i(y_i | \tilde x, \tilde y_i) \, .$
Note that these are strongly related to Bregman distances; if $\wo$ is a saddle point, then $\mathcal F(y | \wo) = D_{f^\ast}^{\mA \xo}(y, \yo)$ is the Bregman distance of $f^\ast$ between $y$ and $y^\opt$. Similarly, we make use of the weighted distance
$$\cF^p(y | \tilde w) \De \sum_{i=1}^n \left(\frac{1}{p_i} - 1\right)\cF_i(y_i | \tilde x, \tilde y_i)$$
and the distance for the primal functional
$\cG(x | \tilde w) \De g(x) - g(\tilde x) - \langle -\mAa \tilde y, x - \tilde x \rangle.$
We note that these distances are also related to the partial primal-dual gap as
$$G_{\sB}(w) = \sup_{\tilde w \in \sB} \cG(x | \tilde w) + \cF(y | \tilde w) = \sup_{\tilde w \in \sB} \cH(w | \tilde w) \, .$$

\begin{theorem} \label{THE:NONSTRCVX}
Let $g, f_i, i = 1, \ldots, n$ be convex, $\theta = 1$ and $\mT, \mS$ be chosen so that there exist ESO parameters $\{v_i\}$ of $\mS^{1/2} \mA \mT^{1/2}$ with
\begin{align}
v_i < p_i \qquad i = 1, \ldots, n \, . \label{EQU:THM1:ESOPARAM}
\end{align}

Then, the Bregman distance between iterates of \cref{ALG:PDHG} $\wk = (\xk, \yk) \in \sW$ and any saddle point $\wo \in \sW$ converges to zero almost surely,
\begin{align}
D^q_h(\wk, \wo) \rightarrow 0 \quad \text{a.s.}  \label{EQU:THM1:ASCONV}%
\end{align}

Moreover, the ergodic sequence $w_{(\Iter)} \De \frac1\Iter \sum_{\iter=1}^\Iter \wk$ converges with rate $1/\Iter$ in an expected partial primal-dual gap sense, i.e. for any set $\sB \De \sB_1 \times \sB_2 \subset \sW$ it holds
\begin{align}
\E \, G_{\sB}(w_{(\Iter)}) \leq \frac{C_\sB}{\Iter} \label{EQU:THM1:ERGGAP}
\end{align}
where the constant is given by
\begin{align}
C_\sB \De \sup_{x \in \sB_1} \frac12 \|\xn - x\|_{\mT^{-1}}^2 + \sup_{y \in \sB_2} \frac12 \|\yn - y\|^2_{\mQ \mS^{-1}} + \sup_{w \in \sB} \cF^p(\yn | w) \, . \label{EQU:THM1:CONST}
\end{align}
The same rate holds for the expected Bregman distance, $\E D_h^q(w_{(\Iter)}, \wo) \leq C_{\{\wo\}} / \Iter$.
\end{theorem}

\begin{remark}
The convergence \cref{EQU:THM1:ASCONV} in Bregman distance implies convergence in norm as soon as $h$ is strictly convex. If $h$ is not strictly convex, then the convergence has to be seen in a more generalized sense. For example, if $h$ is a $\ell^1$-norm (and thus not strictly convex), then the Bregman distance between $\wk$ and $\wo$ is zero if and only if they have the same support and sign. Thus, the convergence statement is related to the support and sign of $\wo$. In the extreme case $h \equiv 0$, then $D_h^q(\cdot, \wo) \equiv 0$ and the convergence statement has no meaning.
\end{remark}

The proof of this theorem utilizes a standard inequality for which we provide the proof in the appendix for completeness.
\begin{lemma} \label{LEM:STD}
Consider the deterministic updates
\begin{align*}
\xkp &= \Prox{\mTk}{g}{\xk - \mTk \mAa \ybk} \\
\yhkp_i &= \Prox{[\mSk]_i}{f^\ast_i}{\yk_i + [\mSk]_i \mA_i \xkp} \qquad i = 1, \ldots, n
\end{align*}
with iteration varying step sizes $\mTk$ and $\mSk = \operatorname{diag}([\mSk]_1, \ldots, [\mSk]_n)$. Then for any $(x, y) \in \sW$ it holds that
\begin{align*}
& \quad \; \|\xk - x\|_{\mTk^{-1}}^2 + \|\yk - y\|^2_{\mSk^{-1}} \\
&\geq \|\xkp - x\|_{\mTk^{-1} + \mu_g\mI}^2 + \|\yhkp - y\|^2_{\mSk^{-1} + \mM} \\
&\qquad + 2 \left(\cG(\xkp | w) + \cF(\yhkp | w) \right) - 2 \langle \mA(\xkp - x), \yhkp - \ybk \rangle \\
&\qquad + \|\xkp - \xk\|_{\mTk^{-1}}^2 + \|\yhkp - \yk\|^2_{\mSk^{-1}}  \, .
\end{align*}
\end{lemma}

\begin{proof}[Proof of \cref{THE:NONSTRCVX}]
The result of \cref{LEM:STD} (with constant step sizes) has to be adapted to the stochastic setting as the dual iterate is updated only with a certain probability. First, a trivial observation is that for any mapping $\varphi$ it holds that
\begin{align}
\varphi(\yhkp_i)
&= \frac1{p_i} \Ekp \varphi(\ykp_i) - \left(\frac1{p_i}-1\right) \varphi(\yk_i)  \notag \\
&= \left(\frac1{p_i}-1\right) \Ekp \varphi(\ykp_i) - \left(\frac1{p_i}-1\right) \varphi(\yk_i) + \Ekp \varphi(\ykp_i) \, . \label{EQU:THM1:TRIVIAL}
\end{align}
Thus, for the generalized distance of $f^\ast$ we arrive at
\begin{align}
\cF(\yhkp | w)
&= \Ekp \cF^p(\ykp | w) - \cF^p(\yk | w) + \Ekp \cF(\ykp | w) \, . \label{EQU:THM1:F}
\end{align}
and for any block diagonal matrix $\mB = \operatorname{diag}(\mB_1, \ldots, \mB_n)$
\begin{align}
  \|\yhkp - \cdot\|^2_\mB &= \Ekp \|\ykp - \cdot\|^2_{\mQ\mB } - \|\yk - \cdot\|^2_{(\mQ - \mI)\mB } \, , \label{EQU:THM1:N} \\
  \yhkp &= \mQ \Ekp \ykp - (\mQ - \mI) \yk \, . \label{EQU:THM1:V}
\end{align}
Using \cref{EQU:THM1:F,EQU:THM1:N,EQU:THM1:V}, we can rewrite the estimate of \cref{LEM:STD} as
\begin{align}
& \quad \; \|\xk - x\|_{\mT^{-1}}^2 + \|\yk - y\|^2_{\mQ \mS^{-1}} + 2 \cF^p(\yk | w) \notag \\
&\geq \Ekp \biggl\{ \|\xkp - x\|_{\mT^{-1}}^2 + \|\ykp - y\|^2_{\mQ \mS^{-1}} + 2 \cF^p(\ykp | w) \notag \\
&\qquad + 2 \cH(\wkp | w) - 2 \Angle{\mA (\xkp - x)}{\mQ (\ykp - \yk) + \yk - \ybk} \notag \\
&\qquad + \|\xkp - \xk\|_{\mT^{-1}}^2 + \|\ykp - \yk\|^2_{\mQ \mS^{-1}} \biggr\} \, . \label{EQU:THM1:START}
\end{align}
where we have used the identity
\begin{align}
\|\cdot\|^2_{\mB} + \|\cdot\|^2_{\mD} = \|\cdot\|^2_{\mB + \mD} \label{EQU:MATRIXID}
\end{align}
to simplify the expression. With the extrapolation $\ybk = \yk + \mQ (\yk - \ykm)$, the inner product term can be reformulated as
\begin{align}
&\quad \; - \Angle{\mA (\xkp - x)}{\mQ (\ykp - \yk) + \yk - \ybk} \notag \\
&= - \Angle{\mQ\mA (\xkp - x)}{\ykp - \yk} + \Angle{\mQ \mA (\xkp - x)}{\yk - \ykm} \notag \\
&= - \Angle{\mQ \mA (\xkp - x)}{\ykp - \yk} + \Angle{\mQ \mA (\xk - x)}{\yk - \ykm} \notag \\
&\qquad + \Angle{\mQ \mA (\xkp - \xk)}{\yk - \ykm} \label{EQU:THM1:IP}
\end{align}
and with \cref{LEM:ESO} and $\gamma^2 \De \max_i v_i / p_i$ it holds that
\begin{align}
&\quad \; 2 \Ek \Angle{\mQ \mA (\xkp - \xk)}{\yk - \ykm} \notag \\
&\geq - \Ek \left\{\gamma^2 \|\xkp - \xk\|^2_{\mT^{-1}} + \|\yk - \ykm\|^2_{\mQ \mS^{-1}} \right\} \, . \label{EQU:THM1:ESO}
\end{align}
Taking expectations with respect to $\sS^1, \ldots, \sS^\Iter$ (denoting this by $\E$) on \cref{EQU:THM1:START}, using the estimates \cref{EQU:THM1:IP,EQU:THM1:ESO} and denoting
\begin{align*}
\Dk \De \E \biggl\{& \|\xk - x\|_{\mT^{-1}}^2 + \|\yk - y\|^2_{\mQ \mS^{-1}} + 2 \cF^p (\yk | w) \\
&\qquad + \|\yk - \ykm\|^2_{\mQ \mS^{-1}} - 2 \Angle{\mQ \mA (\xk - x)}{\yk - \ykm} \biggr\}
\end{align*}
leads with $\gamma < 1$ (follows directly from \cref{EQU:THM1:ESOPARAM}) to
\begin{align}
\Dk &\geq \Dkp + \E \left(2 \cH(\wkp | w) + (1-\gamma^2)\|\xkp - \xk\|_{\mT^{-1}}^2 \right) \notag \\
&\geq \Dkp + 2 \E \cH(\wkp | w) \, . \label{EQU:LEM:STD:STOCHASTIC}
\end{align}
Summing \cref{EQU:LEM:STD:STOCHASTIC} over $\iter = 0, \ldots, \Iter-1$ (note that $y^{(-1)} = y^{(0)}$) and using the estimate (follows directly from \cref{LEM:ESO})
\begin{align*}
    \DK \geq (1-\gamma^2)\|\xK - x\|_{\mT^{-1}}^2 + \|\yK - y\|^2_{\mQ \mS^{-1}} + 2 \cF^p (\yK | w) \geq 2 \cF^p (\yK | w)
\end{align*}
yields
\begin{align}
\E \left\{ \cF^p (\yK | w) + \sum_{\iter=1}^{\Iter} \cH(\wk | w) \right\} \leq \frac{\Dn}{2}\, . \label{LEM:PROOF:SUMMING:EQ}
\end{align}

All assertions of the theorem follow from inequality \cref{LEM:PROOF:SUMMING:EQ}. Inserting a saddle point $w = \wo$ and taking the limit $\Iter \to \infty$, it follows from \cref{LEM:PROOF:SUMMING:EQ} that $\E \sum_{\iter=1}^\infty D_h^q(\wk, \wo) < \infty$ which implies almost surely $\sum_{\iter=1}^{\infty} D_h^q(\wk, \wo) < \infty$ and thus \cref{EQU:THM1:ASCONV}.

To see \cref{EQU:THM1:ERGGAP}, note first that $$\cF^p(\yn | w) - \cF^p(\yK | w) = \cF^p(\yn | x, \yK) \leq \sup_{w \in \sB} \cF^p(\yn | w)$$
and $\frac{\Dn}{2} - \cF^p(\yK | w) \leq C_{\sB}$ if $w \in \sB$ with $C_{\sB}$ as defined in \cref{EQU:THM1:CONST}. Moreover, the generalized distance $\cH(\cdot | w)$ is convex, thus, dividing \cref{LEM:PROOF:SUMMING:EQ} by $\Iter$ yields
\begin{align*}
\E \cH(w_{(\Iter)} |w) \leq \frac 1{\Iter} \E \sum_{\iter=1}^{\Iter} \cH(\wk | w) \leq \frac{C_\sB}{\Iter}
\end{align*}
for any $w \in \sB$. Taking the supremum over $w \in \sB$ yields \cref{EQU:THM1:ERGGAP}. Noting that $D_h^q(w, \wo) = G_{\{\wo\}}(w)$ completes the proof.
\end{proof}
}

\section{Semi-Strongly Convex Case}
In this section we propose two algorithms that converge as $\mathcal O(1/\Iter^2)$ if either $f_i^\ast$ or $g$ is strongly convex. \new{For simplicity we restrict ourselves from now on to scalar-valued step sizes, i.e. $\mT = \tau \mI$ and $\mS_i = \sigma_i \mI$. However, large parts of what follows holds true for operator-valued step sizes, too.}

\begin{algorithm}
\caption{Stochastic Primal-Dual Hybrid Gradient algorithm with acceleration on the primal variable (\textbf{PA-SPDHG}). \newline \textbf{Input:} $\xn, \yn$, $\tn \in \sR, \sn \in \sR^n$, $\sSk$, $\Iter$. \textbf{Initialize:} $\ybn = \yn$} \label{ALG:APDHG:PRIMAL}
\begin{algorithmic}[1]
    \FOR{$\iter = 0, \ldots, \Iter-1$}
    \STATE{$\xkp = \Prox{\tk}{g}{\xk - \tk \mAa \ybk}$}
    \STATE{Select $\sSkp \subset \{1, \ldots, n\}$.}
    \STATE{$\ykp_i = \begin{cases} \Prox{\sk_i}{f^\ast_i}{\yk_i + \sk_i \mA_i \xkp} & \text{if $i \in \sSkp$} \\ \yk_i & \text{else} \end{cases}$}
    \STATE{$\thk = (1 + 2 \mu_g \tk)^{-1/2}\, , \quad \tkp = \thk \tk \, , \quad \skp = \sk / \thk$}
    \STATE{$\ybkp = \ykp + \thk \mQ \left(\ykp - \yk\right)$}
    \ENDFOR
\end{algorithmic}
\end{algorithm}

\begin{algorithm}
\caption{Stochastic Primal-Dual Hybrid Gradient algorithm with acceleration on the dual variable (\textbf{DA-SPDHG}). \newline \textbf{Input:} $\xn, \yn$, $\tn \in \sR, \stn \in \sR$, $\sSk$, $\Iter$. \textbf{Initialize:} $\ybn = \yn$} \label{ALG:APDHG:DUAL}
\begin{algorithmic}[1]
    \FOR{$\iter = 0, \ldots, \Iter-1$}
    \STATE{$\xkp = \Prox{\tk}{g}{x^\iter - \tk \mAa \ybk}$}
    \STATE{Select $\sSkp \subset \{1, \ldots, n\}$.}
    \STATE{$\sk_i = \frac{\stk}{\mu_i [p_i - 2(1-p_i) \stk]} \, , \quad i \in \sSkp$}
    \STATE{$\ykp_i = \begin{cases} \Prox{\sk_i}{f^\ast_i}{\yk_i + \sk_i \mA_i \xkp} & \text{if $i \in \sSkp$} \\ \yk_i & \text{else} \end{cases}$}
    \STATE{$\thk = (1 + 2 \stk)^{-1/2} \, , \quad \tkp = \tk / \thk \, , \quad \stkp = \thk \stk$}
    \STATE{$\ybkp = \ykp + \thk \mQ \left(\ykp - \yk\right)$}
    \ENDFOR
\end{algorithmic}
\end{algorithm}

\begin{theorem}[Dual Strong Convexity] \label{THM:SEMISTRONGLY:DUAL}
Let $f_i^\ast$ be strongly convex with constants $\mu_i > 0, i = 1, \ldots, n$. Consider \cref{ALG:APDHG:DUAL} and let the initial step sizes $\stn, \tn$ be chosen such that
\begin{align}
\stn < \min_i \frac{p_i}{2 (1 - p_i)} \, . \label{EQU:SIGMATILDE}
\end{align}
and for the ESO parameters $\{v_i\}$ of $\mSn^{1/2} \mA \tn^{1/2}$ it holds that
\begin{align}
v_i \leq p_i \qquad i = 1, \ldots, n \label{EQU:THM2:ESOPARAM}
\end{align}
with $[\mSk]_i = \sk_i \mI$ and
\begin{align}
\sk_i = \frac{\stk}{\mu_i [p_i - 2(1-p_i) \stk]} \, .
\end{align}
Then there exists $\tilde \Iter \in \sN$ such that for all $\Iter \geq \tilde \Iter$ it holds
\begin{align*}
\E \|\yK - \yo\|_{\mYn}^2 \leq \frac 2{\Iter^2} \left( \|\xn - \xo\|_{\tn^{-1}}^2 + \|\yn - \yo\|_{\mYn}^2 \right)
\end{align*}
where the metric on $\sY$ is defined by $\mYk \De \mQ \mSk^{-1} + 2 \mM(\mQ - \mI)$.
\end{theorem}

\begin{remark}
As already noted in \cite{Chambolle2011}, $\tilde \Iter$ is usually fairly small so that the estimate in \cref{THM:SEMISTRONGLY:DUAL} has practical relevance.
\end{remark}

\begin{remark}
For serial sampling the condition on the ESO parameters \cref{EQU:THM2:ESOPARAM} is equivalent to
\begin{align*}
    \stn \leq \min_i \frac{\mu_i p_i^2}{\tn \|\mA_i\|^2 + 2 \mu_i p_i(1-p_i)} \, .
\end{align*}
In particular, it implies condition \cref{EQU:SIGMATILDE} on $\stn$.
\end{remark}

\begin{remark}
The convergence of \cref{ALG:APDHG:PRIMAL} with acceleration on the primal variable is similar to the deterministic case, cf. Appendix C.2 of \cite{Chambolle2016}, and omitted here for brevity. It converges with rate $\mathcal O(1/K^2)$ if the ESO parameters satisfy $v_i < p_i$.
\end{remark}

This theorem requires an estimate on the expected contraction similar to the proof of \cref{THE:NONSTRCVX} and shown in the appendix.

\new{\begin{lemma} \label{LEM:STD:STOCHASTIC}
Let $\xkp, \yhkp$ be defined as in \cref{LEM:STD} and $\ykp_i = \yhkp_i$ with probability $p_i > 0$ and unchanged else. Moreover, let
\begin{align}
\ybkp = \ykp + \thk \mQ \left(\ykp - \yk\right) \label{EQU:LEM:EXTRA}
\end{align}
and $\{v_i\}$ be some ESO parameters of $\mSk^{1/2} \mA \tk^{1/2}$. Then with $\gamma^2 = \max_i \frac{v_i}{p_i}$ it holds
\begin{align*}
& \quad \; \E^{(\iter, \iter-1)} \biggl\{ \|\xk - \xo\|_{\tk^{-1}}^2 + \|\yk - \yo\|^2_{\mQ \mSk^{-1} + 2 \mM(\mQ - \mI)} \\
&\qquad - 2 \thkm \langle \mQ \mA(\xk - \xo), \yk - \ykm \rangle + (\gamma \thkm)^2  \|\yk - \ykm\|^2_{\mQ \mSk^{-1}} \biggr\} \\
&\geq \E^{(\iter+1, \iter)} \biggl\{ \|\xkp - \xo\|_{\tk^{-1} + 2 \mu_g\mI}^2 + \|\ykp - \yo\|^2_{\mQ\mSk^{-1} + 2 \mM\mQ} \\
&\qquad - 2 \langle \mQ \mA(\xkp - \xo), \ykp - \yk \rangle + \|\ykp - \yk\|^2_{\mQ \mSk^{-1}} \biggr\} \, .
\end{align*}
\end{lemma}}

\begin{proof}[Proof of \cref{THM:SEMISTRONGLY:DUAL}]
The update on the step sizes in \cref{ALG:APDHG:DUAL} imply that
\begin{align}
    \thk \frac{1}{\tk} &\geq \frac{1}{\tkp} \, , \notag \\
    \thk \left(\frac{1}{p_i\sk_i} + \frac{2 \mu_i}{p_i}\right) &\geq \frac{1}{p_i \skp_i} + \frac{2  (1 - p_i) \mu_i}{p_i} \label{THM:SEMISTRONGLY:DUAL:STEPS:B}
\end{align}
for all $i = 1, \ldots, n$ and therefore
\begin{align}
    \thk \|\cdot\|_{\tk^{-1}}^2 &\geq \|\cdot\|^2_{\tkp^{-1}} \, , \label{THM2:STEP1} \\
    \thk \|\cdot\|_{\mQ \mSk^{-1} + 2 \mM \mQ}^2 &\geq \|\cdot\|_{\mQ \mSkp^{-1} + 2 \mM (\mQ - I)}^2 = \|\cdot\|_{\mYkp}^2 \, . \label{THM2:STEP2}
\end{align}
To see \cref{THM:SEMISTRONGLY:DUAL:STEPS:B}, the auxiliary sequence $\stk$ satisfies
\begin{align*}
\stk = \frac{p_i \mu_i \sk_i}{1 + 2(1-p_i)\mu_i \sk_i}
\end{align*}
such that \cref{THM:SEMISTRONGLY:DUAL:STEPS:B} is satisfied as soon as
\begin{align}
    \thk \frac{1 + 2 \stk}{\stk} \geq \frac 1{\stkp} \, . \label{THM:SEMISTRONGLY:DUAL:CONDITION:DUAL:TILDE}
\end{align}
Note that the transformation from $\stk$ to $\sk_i$ is well-defined if $\stk < \min_i \frac{p_i}{2 (1-p_i)}$ which is the case as $\stk$ is monotonically non-increasing and $\stn$ satisfies the condition. By construction of the sequence $\stkp = \thk \stk$, \cref{THM:SEMISTRONGLY:DUAL:CONDITION:DUAL:TILDE} is solved with equality by $\thk = (1 + 2\stk)^{-1/2}$. Moreover, the sequence $\sk_i$ is also non-increasing as
\begin{align*}
    \skp_i &= \frac{\thk \sk_i}{1 + 2 (1-\thk)(1-p_i) \mu_i \sk_i} \leq \thk \sk_i \, ,
\end{align*}
thus, with \cref{EQU:THM2:ESOPARAM} we see that the ESO parameters of $\mSk^{1/2} \mA \tk^{1/2}$ are also bounded by $p_i$.

For the actual proof of the theorem, note that the inequalities \cref{THM2:STEP1,THM2:STEP2} imply
\begin{align}
& \quad \; \thk \E \biggl\{ \|\xkp - \xo\|_{\tk^{-1}}^2 + \|\ykp - \yo\|_{\mQ \mSk^{-1} + 2\mM \mQ}^2 \notag \\
& \quad - 2 \Angle{\mQ \mA(\xkp - \xo)}{\ykp - \yk} \biggr\}
\geq \E \Dkp \label{THM:SEMISTRONGLY:DUAL:INEQ1}
\end{align}
with
$$\Dk \De \|\xk - \xo\|_{\tk^{-1}}^2 + \|\yk - \yo\|_{\mYk}^2 - 2 \thkm \Angle{\mQ \mA(\xk - \xo)}{\yk - \ykm} \, .$$
Thus, combining \cref{LEM:STD:STOCHASTIC} ($\mu_g = 0$) and \cref{THM:SEMISTRONGLY:DUAL:INEQ1} yields
\begin{align*}
& \quad \; \thk \E \left\{ \Dk + (\gamma\thkm)^2 \|\yk - \ykm\|^2_{\mQ \mSk^{-1}} \right\} \\
&\geq \thk \E \biggl\{ \|\xkp - \xo\|_{\tk^{-1}}^2 + \|\ykp - \yo\|^2_{\mQ\mSk^{-1} + 2 \mM\mQ} \\
&\qquad - 2 \langle \mQ \mA(\xkp - \xo), \ykp - \yk \rangle + \|\ykp - \yk\|^2_{\mQ \mSk^{-1}} \biggr\} \\
&\geq \E \left\{ \Dkp + \thk \|\ykp - \yk\|^2_{\mQ \mSk^{-1}}\right\} \, .
\end{align*}
With $\gamma \thkm \leq 1, \mSkp \leq \thk \mSk$ and $\bar \Delta^{(\iter)} \De \E \Bigl\{\Dk + \|\yk - \ykm\|^2_{\mQ \mSk^{-1}}\Bigr\}$ we derive the recursion
\begin{align*}
\thk \bar \Delta^{(\iter)}
&\geq \thk \E \left\{ \Dk + (\gamma\thkm)^2 \|\yk - \ykm\|^2_{\mQ \mSk^{-1}} \right\} \\
&\geq \E \left\{ \Dkp + \thk \|\ykp - \yk\|^2_{\mQ \mSk^{-1}}\right\} \geq \bar \Delta^{(\iter+1)}  \, .
\end{align*}
Using this inequality recursively, $y^{(-1)} = \yn$, we arrive at
\begin{align*}
\prod_{\iter=0}^{\Iter-1} \thk \bar \Delta^{(0)}
\geq \bar \Delta^{(\Iter)}
&\geq \E \left\{ (1 - \gamma^2)\|\xK - \xo\|_{\tK^{-1}}^2 + \|\yK - \yo\|_{\mYK}^2 \right\} \\
&\geq \E \|\yK - \yo\|_{\mYK}^2
\end{align*}
where the second estimated follows directly from \cref{LEM:ESO} and the third inequality from $\gamma \leq 1$ which holds by assumption \cref{EQU:THM2:ESOPARAM}.

As $\bar \Delta^{(0)} = \|\xn - \xo\|_{\tn^{-1}}^2 + \|\yn - \yo\|_{\mYn}^2$, $\thk = \frac{\stkp}{\stk}$ and $$\|\cdot\|_{\mYK}^2 = \frac{1}{\stK}\|\cdot\|^2_{\mM} = \frac{\stn}{\stK}\|\cdot\|^2_{\mYn}$$ which holds by the definition of $\stk$, it holds that
\begin{align*}
\E \|\yK - \yo\|_{\mYn}^2 \leq \left(\frac{\stK}{\stn}\right)^2 \left\{ \|\xn - \xo\|_{\tn^{-1}}^2 + \|\yn - \yo\|_{\mYn}^2 \right\} \, .
\end{align*}
Finally, the assertion follows by Corollary 1 of \cite{Chambolle2011}.
\end{proof}

\section{Strongly Convex Case}
If both $f_i^\ast$ and $g$ are strongly convex, we may find step size parameters such that the \cref{ALG:PDHG} converges linearly.
\begin{theorem} \label{THM:STRONGLY}
Let $(x^\opt, y^\opt) \in \sW$ be a saddle point and $g, f_i^\ast$ be strongly convex with constants $\mu_g, \mu_i > 0, i = 1, \ldots, n$. Let the step sizes $\tau, \sigma_1, \ldots, \sigma_n, 0 < \theta < 1$ be chosen such that the ESO parameters $\{v_i\}$ of $\mS^{1/2} \mA \tau^{1/2}$ can be estimated as
\begin{align}
    v_i &< \frac{p_i}{\theta} \qquad i = 1, \ldots, n, \label{EQU:THM3:ESOPARAM}
\end{align}
and the extrapolation $\theta$ satisfies the lower bounds
\begin{align}
    \theta \geq \frac{1}{1 + 2 \mu_g \tau} \, , \qquad \theta \geq \frac{1 + 2 (1-p_i) \mu_i \sigma_i}{1 + 2 \mu_i \sigma_i} \qquad i = 1, \ldots, n \, .\label{EQU:THM:STRONGLY:STEPSIZE}
\end{align}

Then the iterates of \cref{ALG:PDHG} converge linearly to the saddle point, in particular
\begin{align*}
\E \left\{ (1-\gamma^2 \theta) \|\xK - \xo\|_{\mX}^2 + \|\yK - \yo\|_{\mY}^2 \right\} \leq \theta^\Iter \left\{ \|\xn - \xo\|_{\mX}^2 + \|\yn - \yo\|_{\mY}^2 \right\}
\end{align*}
holds where the metrics are given by $\mX \De (\tau^{-1} + 2 \mu_g) \mI$, $\mY \De (\mS^{-1} + 2 \mM) \mQ$ and $\gamma^2 = \max_i v_i / p_i$.
\begin{proof}
The requirements \cref{EQU:THM:STRONGLY:STEPSIZE} on the step sizes $\tau, \sigma_1, \ldots, \sigma_n$ and $\theta$ imply $\theta \|\cdot\|_{\mX}^2 \geq \|\cdot\|^2_{\tau^{-1}}$ and $\theta \|\cdot\|_{\mY}^2 \geq \|\cdot\|^2_{\mQ \mS^{-1} + 2 \mM (\mQ - \mI)}$. Thus, we directly get
\begin{align}
\theta \, \E \Dk &\geq \E \biggl\{ \|\xk - \xo\|_{\tau^{-1}}^2 + \|\yk - \yo\|^2_{\mQ \mS^{-1} + 2 \mM (\mQ - \mI)} \notag \\
&\qquad - 2 \theta \Angle{\mQ \mA(\xk - \xo)}{\yk - \ykm} \biggr\} \, . \label{EQU:STRONGLY:INEQ2}
\end{align}
where we denoted
$$\Dk \De \|\xk - \xo\|_{\mX}^2 + \|\yk - \yo\|_{\mY}^2 - 2 \Angle{\mQ \mA (\xk - \xo)}{\yk - \ykm} \, .$$
Combining \cref{EQU:STRONGLY:INEQ2,LEM:STD:STOCHASTIC} with constant step sizes yields
\begin{align*}
\theta \, \E \Dk \,
&\geq \E \left\{ \Dkp + \|\ykp - \yk\|^2_{\mQ \mS^{-1}} - (\gamma \theta)^2 \|\yk - \ykm\|^2_{\mQ \mS^{-1}} \right\} \, .
\end{align*}
Multiplying both sides by $\theta^{-(\iter+1)}$ and summing over $k = 0, \ldots, \Iter-1$ yields
\begin{align*}
\Dn
&\geq \theta^{-\Iter} \E \Bigl\{ \DK + \|\yK - \yKm\|^2_{\mQ \mS^{-1}} \Bigr\} \\
&\quad + (1 - \gamma^2\theta) \E \sum_{\iter=1}^{\Iter-1} \theta^{-\iter} \|\yk - \ykm\|^2_{\mQ \mS^{-1}} \\
&\geq \theta^{-\Iter} \E \biggl\{ \|\xK - \xo\|_{\mX}^2 + \|\yK - \yo\|_{\mY}^2 + \|\yK - \yKm\|^2_{\mQ \mS^{-1}} \\
&\qquad - 2 \Angle{\mQ \mA(\xK - \xo)}{\yK - \yKm} \biggr\} \\
&\geq \theta^{-\Iter} \E \left\{ \|\xK - \xo\|_{\mX}^2 - \gamma^2\|\xK - \xo\|_{\tau^{-1}}^2 + \|\yK - \yo\|_{\mY}^2 \right\} \\
&\geq \theta^{-\Iter} \E \left\{ (1-\gamma^2\theta) \|\xK - \xo\|_{\mX}^2 + \|\yK - \yo\|_{\mY}^2 \right\}
\end{align*}
where we used again Lemma \ref{LEM:ESO} and the non-negativity of norms for the second inequality. Thus, the assertion is proven.
\end{proof}
\end{theorem}

\subsection{Optimal Parameters for Serial Sampling} \label{SEC:PARAM}
This analysis is to optimize the convergence rate $\theta$ of \cref{THM:STRONGLY} for three different serial sampling options where exactly one block is chosen in each iteration. Other sampling strategies, including multi-block, parallel, etc. \cite{Qu2015} will be subject of future work.

We will derive the rates and parameters in terms of the \emph{condition numbers} $\kappa_i \De \|\mA_i\|^2/(\mu_g \mu_i)$ as these are scaling invariant, thus we cannot improve the rates by simple rescaling of the problem. This can be seen as follows. If we rewrite problem \cref{EQU:SADDLE} in terms of the scaled variables $\overline x \De \alpha x$ and $\overline y_i \De \beta_i y_i$, then the corresponding operators $\overline \mA_i \De \mA_i / (\alpha \beta_i)$ have norm $\|\overline \mA_i\| = \|\mA_i\|/(\alpha \beta_i)$, the function $\overline g(\overline x) \De g(\overline x / \alpha)$ is $\overline \mu_g \De \mu_g / \alpha^2$ strongly convex and the functions $\overline f^\ast_i(\overline y_i) \De f^\ast_i(\overline y_i/\beta_i)$ are $\overline \mu_i \De \mu_i/\beta_i^2$ strongly convex. Thus the condition numbers are scaling invariant as
\begin{align*}
\overline \kappa_i
= \frac{\|\overline \mA_i\|^2}{\overline{\mu}_g \overline{\mu}_i }
= \frac{\frac1{(\alpha \beta_i)^2} \|\mA_i\|^2}{\frac1{\alpha^2} \mu_g \frac1{\beta_i^2} \mu_i}
= \frac{\|\mA_i\|^2}{\mu_g \mu_i}
= \kappa_i \, .
\end{align*}

With $\bar \sigma_i \De \sigma_i\mu_i$ and $\bar \tau \De \tau\mu_g$, the conditions on the step sizes \cref{EQU:THM:STRONGLY:STEPSIZE} become
\begin{align}
\theta \ge \frac{1}{1+2\bar\tau} \, , \quad
\theta \ge \max_i 1-2\frac{\bar\sigma_i p_i}{1+2\bar\sigma_i} \, , \quad \text{and} \quad
\max_i \bar\tau\bar\sigma_i \kappa_i \theta \le \rho^2 p_i \label{EQU:STEPSIZECONDITION:SERIAL}
\end{align}
for some $\rho<1$. The last condition arises from the ESO parameters of serial sampling which are $v_i = \sigma_i \tau \|\mA_i\|^2$, see \cref{EXA:SERIAL}. Finding optimal parameters is equivalent to equating the above inequalities. Note that the first two conditions (with equality) are equivalent to $\theta \bar\tau = (1 - \theta)/2$ and $\bar \sigma_i = \frac{1 - \theta}{2 (p_i - (1 - \theta))}$.
With these choices, the third condition in \cref{EQU:STEPSIZECONDITION:SERIAL} reads
\begin{align}
(1 - \theta)^2 \kappa \leq 4 \rho^2 p_i (p_i-(1-\theta)) \qquad i=1, \ldots, n \, . \label{EQU:STEPSIZECONDITION}
\end{align}
It follows from \cref{EQU:STEPSIZECONDITION} that with $\tilde\kappa = 1 + \kappa / \rho^2$ it holds
\begin{align}
\theta \geq \max_i 1 - \frac{2 p_i}{1 + \sqrt{\tilde\kappa_i}} \label{EQU:THETA}
\end{align}

\begin{example}[Serial uniform sampling] We first consider uniform sampling, i.e. every block is sampled with the same probability $p_i = 1 / n$. Then it is easy to see that the smallest achievable rate is given by
\begin{align}
\theta_{\text{uni}} = 1 - \frac{2}{n + n \max_j \sqrt{\tilde\kappa_j}} \label{EQU:PARAMTERS:UNIFORM}
\end{align}
and the step sizes become
$$ \sigma_i = \frac{\mu_i^{-1}}{\max_j \sqrt{\tilde\kappa_j} - 1} \, ,
\quad \tau = \frac{\mu_g^{-1}}{n - 2 + n \max_j \sqrt{\tilde\kappa_j}} \, .$$
\end{example}

\begin{example}[Serial importance sampling] Instead of uniform sampling we may sample \enquote{important blocks} more often, i.e. we sample every block with a probability proportional to the square root of its condition number $p_i = \sqrt{\kappa_i}/\sum_j \sqrt{\kappa_j}$. Then the smallest rate that achieves \cref{EQU:THETA} is given by
\begin{align}
\theta_{\text{imp}} = 1 - \frac{2 \nu}{\sum_{j=1}^n \sqrt{\kappa_j}} \label{EQU:PARAMTERS:IMPORTANCE}
\end{align}
with $\nu \De \min_j \sqrt{\kappa_j} / (1 + \sqrt{\tilde \kappa_j})$ and the step sizes are
$$\sigma_i = \frac{\nu \mu_i^{-1}}{\sqrt{\kappa_i} - 2 \nu} \, ,
\quad \tau = \frac{\nu \mu_g^{-1}}{\sum_{j=1}^n \sqrt{\kappa_j} - 2 \nu} \, .$$
\end{example}

\begin{example}[Serial optimal sampling] Instead of a predefined probability we will seek for an \enquote{optimal sampling} that minimizes the linear convergence rate $\theta$. The optimal sampling can be found by equating condition \cref{EQU:THETA} for $i = 1, \ldots, n$
\begin{align}
    \theta \left(1 + \sqrt{\tilde\kappa_i} \right) = 1 + \sqrt{\tilde \kappa_i} - 2 p_i \label{EQU:OPTIMALPARAMETERS:1} \, .
\end{align}
Summing \cref{EQU:OPTIMALPARAMETERS:1} from 1 to $n$ and using that for serial sampling $\sum_{i=1}^n p_i = 1$ leads to
\begin{align}
\theta_{\text{opt}} = 1 - \frac{2}{n + \sum_{j=1}^n \sqrt{\tilde\kappa_j}} \label{EQU:PARAMTERS:OPTIMAL}
\end{align}
with step size parameters
$$ \sigma_i = \frac{\mu_i^{-1}}{\sqrt{\tilde\kappa_i} - 1} \, ,
\quad \tau = \frac{\mu_g^{-1}}{n - 2 + \sum_{j=1}^n \sqrt{\tilde\kappa_j} } $$
and probabilities $$p_i = \frac{1 + \sqrt{\tilde\kappa_i}}{n + \sum_{j=1}^n \sqrt{\tilde\kappa_j}} \, .$$
\end{example}

\begin{remark}[Minibatches]
All arguments above can readily be extended to samplings where at each iteration not only one but a fixed number of blocks are chosen.
\end{remark}

\begin{remark}[Better Sampling]
It is easy to see that optimal sampling is better than uniform sampling: if all condition numbers are the same, then the rates for uniform sampling \cref{EQU:PARAMTERS:UNIFORM} and optimal sampling \cref{EQU:PARAMTERS:OPTIMAL} are equal but if they are not, then the rate of optimal sampling is strictly smaller and thus better.

Moreover, optimal sampling is better than importance sampling. To see this, note that due to the monotonicity of $ \sqrt{x} / (1 + \sqrt{1 + x})$ we get
\begin{align*}
\theta_{\operatorname{imp}}
&= 1 - \min_i \frac{2}{\left(1 + \sqrt{\tilde\kappa_i}\right) \sum_{j=1}^n \sqrt{\kappa_j/\rho^2} / \sqrt{\kappa_i/\rho^2}} \\
&\geq 1 - \min_i \frac{2}{\left(1 + \sqrt{\tilde\kappa_i} \right) \sum_{j=1}^n (1 + \sqrt{\tilde \kappa_j}) / (1 + \sqrt{\tilde \kappa_i})}
= \theta_{\operatorname{opt}} \, .
\end{align*}
\end{remark}

\begin{remark}[Comparison to Zhang and Xiao \cite{Zhang2015}]
The algorithm of Zhang and Xiao \cite{Zhang2015} is (almost\footnote{In contrast to our work, they have an extrapolation on both primal and dual variables. However, both extrapolations are related as our extrapolation factor is the product of their extrapolation factors.}) a special case of the proposed algorithm where each block is picked with probability $p_i = 1/n$. Here $m$ denotes the size of each block to be processed at every iteration and $n$ the number of blocks. Moreover, they only consider the strongly convex case where $g$ is $\mu_g$-strongly convex and all $f_i^\ast$ are $\mu_f$-strongly convex. Then with $R$ being the largest norm of the rows in $\mA$ they achieve
\begin{align*}
    \theta_{\text{ZX}} = 1 - \frac{1}{n + n \frac{\sqrt{m} R}{\sqrt{\mu_g \mu_f}}} \, .
\end{align*}
If the minibatch size is $m = 1$, the blocks are chosen to be single rows and the probabilities are uniform, then their rate is slightly worse than ours
\begin{align*}
    \theta_{\text{ZX}}
    = 1 - \frac{1}{n + n \max_j \sqrt{\kappa_j}}
    &\geq 1 - \frac{2}{2 n + n \max_j \sqrt{\kappa_j/\rho^2}} \\
    &\geq 1 - \frac{2}{2 n + n (\max_j \sqrt{1 + \kappa_j / \rho^2} - 1)}
    = \theta_{\operatorname{uni}}
\end{align*}
for any $\rho \geq \frac{1}{2}$. For $m > 1$, the rates differ even more as the condition numbers are conservatively estimated. Similarly, the rates can be improved by non-uniform sampling if the row norms are not equal.
\end{remark}

\section{Numerical Results}
All numerical examples are implemented in python using numpy and the operator discretization library (ODL) \cite{Adler2017}. {{\color{red}The python code and all example data will be made available on github upon acceptance of this manuscript.}

\subsection{Non-Strongly Convex PET Reconstruction}
In this example we consider positron emission tomography (PET) reconstruction with a total variation (TV) prior. The goal in PET imaging is to reconstruct the distribution of a radioactive tracer from its line integrals \cite{Ollinger1997a}. Let $\sX = \sR^{d_1 \times d_2}, d_1 = d_2 = 250$ be the space of tracer distributions (images) and $\sY_i = \sR^{|\sB_i|}$ the data spaces where $\sB_i \subset \{1, \ldots, N\}, N = 200 \cdot 250$ (200 views around the object) are subsets of indices with $\sB_i \cap \sB_j = \emptyset$ if $i \neq j$ and $\cup_{i=1}^n \sB_i = \{1, \ldots, N\}$. All samplings in this example divide the views equidistantly. It is standard that PET reconstruction can be posed as the optimization problem \cref{EQU:MINPROB} where the data fidelity term is given by the Kullback--Leibler divergence
\begin{align}
    f_i(y) =
    \begin{cases}
    \sum_{j \in \sB_i} y_j + r_j - b_j + b_j \log\left(\frac{b_j}{y_j + r_j}\right) & \text{if $y_j + r_j > 0$} \\
    \infty & \text{else}
    \end{cases} \label{EQU:EXAMPLES:KL}
\end{align}
where it is convention that $0 \log 0 \De 0$. The operator $\mA$ is a scaled X-ray transform where in each of 200 directions 250 line integrals are computed with the astra toolbox \cite{VanAarle2015,VanAarle2016}. The prior is the TV of $x$ with non-negativity constraint, i.e. $g(x) = \alpha \|\nabla x\|_{1,2} + \chi_{\geq 0}(x)$, with regularization parameter $\alpha = 0.2$ and the gradient operator $\nabla x = (\nabla_1 x, \nabla_2 x) \in \sR^{d_1 \cdot d_2 \times 2}$ is discretized by forward differences in horizontal and vertical direction, cf. \cite{Chambolle2004} for details.
The $1,2$-norm of these gradients is defined as $\|x\|_{1,2} \De \sum_j \sqrt{(\nabla_1 x_j)^2 + (\nabla_2 x_j)^2}$. The Fenchel conjugate of the Kullback--Leibler divergence \cref{EQU:EXAMPLES:KL} is
\begin{align}
    f_i^\ast(z) = \sum_{j \in \sB_i}
    \begin{cases}
    -z_j r_j - b_j \log(1-z_j) & \text{if $z_j \leq 1$ and ($b_j = 0$ or $z_j < 1$)} \\
    \infty & \text{else}
    \end{cases} \label{EQU:EXAMPLES:KLCONJ} \, ,
\end{align}
its proximal operator given by
\begin{align*}
    \left[\Prox{\sigma_i}{f^\ast_i}{z}\right]_j &= \frac12\left(z_j + 1 + \sigma_i r_j - \sqrt{(z_j - 1 + \sigma_i r_j)^2 + 4\sigma_i b_j} \right).
\end{align*}
The proximal operator for $g$ is approximated with 20 iterations of the fast gradient projection method (FGP) \cite{Beck2009a} with a warm start applied to the dual problem.

\def\StatsHeight{4.5cm}
\begin{figure}%
\def\ImWidth{3cm}%
\centering%
\includegraphics[height=\StatsHeight]{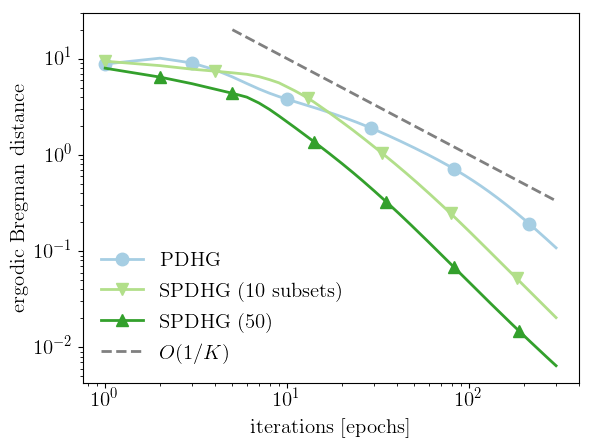} \hspace*{5mm}%
\includegraphics[height=\StatsHeight]{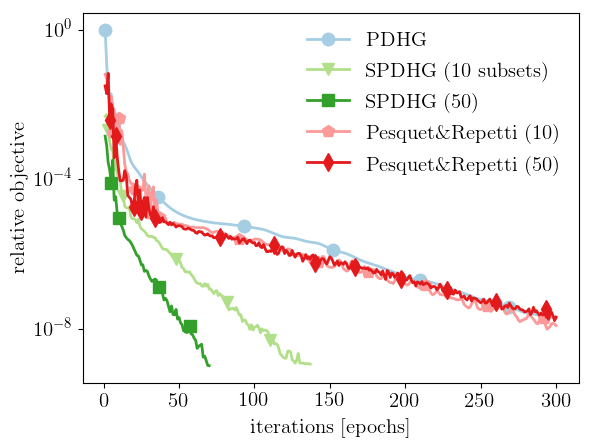}%
\caption{PET reconstruction with TV solved as a non-strongly convex problem. \textbf{Left}: As proven in \cref{THE:NONSTRCVX}, the ergodic Bregman distances converge indeed with rate $O(1/\Iter)$. \textbf{Right}: Speed comparison measured in terms of relative objective $[\Phi(\xK) - \Phi(\xo)] / [\Phi(\xn) - \Phi(\xo)]$. The proposed algorithm SPDHG converges faster than the algorithm of Pesquet\&Repetti \cite{Pesquet2015} and the deterministic PDHG.} \label{fig:case1:PETTV}%
\vspace*{4mm}
\includegraphics[width=\ImWidth]{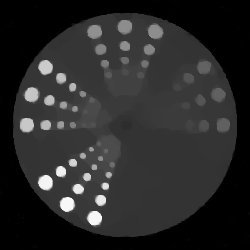}\hspace{5mm}%
\includegraphics[clip, trim=.1cm .1cm .1cm .1cm, width=\ImWidth]{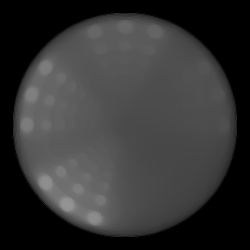}\hspace{1mm}%
\includegraphics[clip, trim=.1cm .1cm .1cm .1cm, width=\ImWidth]{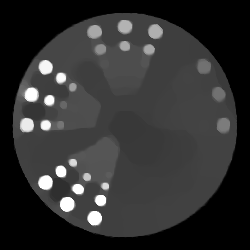}\hspace{1mm}%
\includegraphics[clip, trim=.1cm .1cm .1cm .1cm, width=\ImWidth]{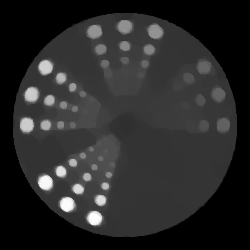}%
\caption{PET reconstruction results after 5 epochs with uniform sampling of 50 subsets. \textbf{From left to right:} approximate primal part of saddle point, PDHG, Pesquet\&Repetti \cite{Pesquet2015} and SPDHG. With the same number of operator evaluations both stochastic algorithms make much more progress towards the saddle point.} \label{fig:case1:PETTV:images}%
\end{figure}%

\paragraph{Parameters} In this experiment we choose $\gamma = 0.99$, $\theta = 1$ and all samplings are uniform, i.e. $p_i=1/n$. The number of subsets varies between $n=1$ (deterministic case), 50 and 250. The other step size parameters are chosen as
\begin{itemize}
 \item PDHG, Pesquet\&Repetti \cite{Pesquet2015}: $\sigma_i = \tau = \gamma/\|\mA\| \approx 6.9 \cdot 10^{-4}$
 \item SPDHG: $\sigma_i = \gamma/\|\mA_i\| \approx 2.2 \cdot 10^{-3}$, $\tau = \gamma / (n \max_i \|\mA_i\|) \approx 2.2 \cdot 10^{-4}$
\end{itemize}

\paragraph{Results}
\cref{fig:case1:PETTV} on the left shows that the ergodic Bregman distance converges with rate $1/k$ as proven in \cref{THE:NONSTRCVX}. On the right we compare the deterministic PDHG with the randomized SPDHG and the algorithm of Pesquet\&Repetti. It can be clearly seen that the proposed SPDHG converges much faster than both the algorithm of Pesquet\&Repetti and the deterministic PDHG. Some example images are found in \cref{fig:case1:PETTV:images} after 5 epochs which again highlight the speed-up gained by randomization.

\subsection{TV denoising with Gaussian Noise (Primal Acceleration)}
\begin{figure}%
\def\ImWidth{4cm}%
\centering%
\includegraphics[height=\StatsHeight]{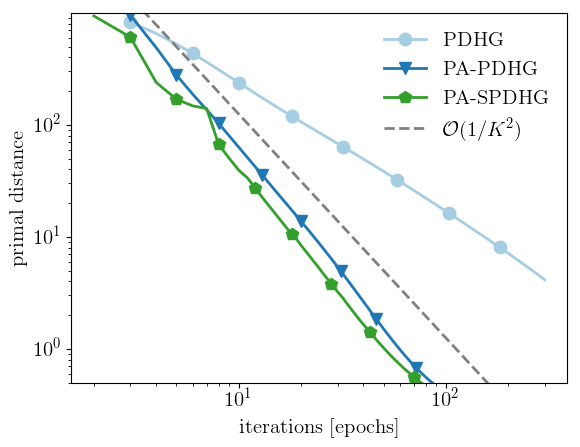} \hspace*{5mm}%
\includegraphics[height=\StatsHeight]{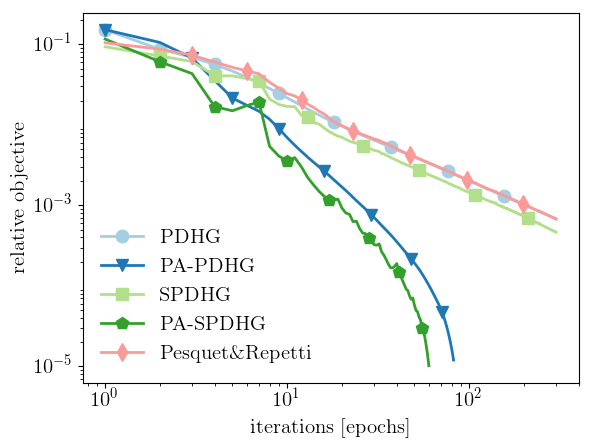}%
\caption{Primal acceleration for TV denoising. \textbf{Left:} Primal distance to saddle point $\|\xK - \xo\|^2$ \textbf{Right:} Relative objective $[\Phi(\xK) - \Phi(\xo)] / [\Phi(\xn) - \Phi(\xo)]$.} \label{fig:case2:L2TV:quant}%
\vspace*{2mm}
\includegraphics[width=\ImWidth]{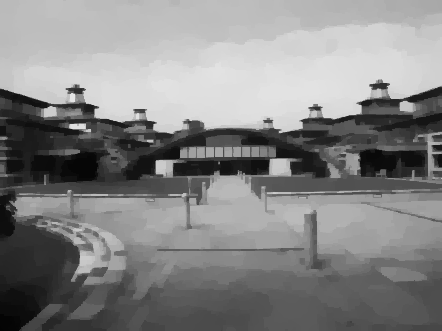}\hspace*{5mm}%
\includegraphics[width=\ImWidth]{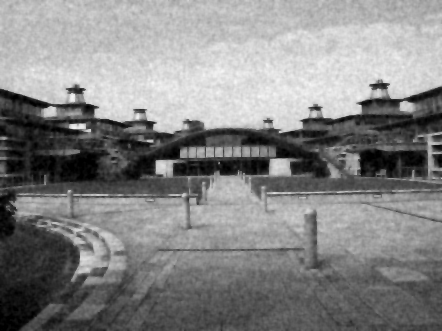}\hspace*{1mm}%
\includegraphics[width=\ImWidth]{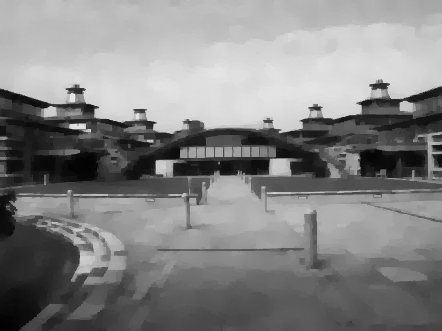}%
\caption{TV denoised images. \textbf{Left:} Approximate primal part of saddle point $\xo$ after 2000 PDHG iterations. \textbf{Right:} PDHG and primal accelerated SPDHG (PA-SPDHG) after 20 epochs.} \label{fig:case2:L2TV:images}%
\end{figure}%

In the second example we consider denoising of an image that is degraded by Gaussian noise with the help of the anisotropic TV. This can be achieved by solving \cref{EQU:MINPROB} with $\sX = \sR^{d_1 \times d_2}$, $d_1 = 442$, $d_2 = 331$, the data fit $g(x) = 1 / (2\alpha) \|x - b\|^2_2$ is the squared Euclidean norm and the prior the (anisotropic) TV $f_i(y_i) = \|y_i\|_1, \mA_i = \nabla_i$ and $n = 2$. Instead of the isotropic TV as in the previous example we consider here the anisotropic version as it is separable in the direction of the gradient. The regularization parameter is chosen to be $\alpha = 0.12$. See e.g. \cite{Chambolle2011} for details on convex conjugates and proximal operators of these functionals.
\paragraph{Parameters} In this experiment we choose $\gamma = 0.99$ and the sampling to be uniform, i.e. $p_i = 1/n$. The number of subsets are either $n=1$ in the deterministic case or $n=2$ in the stochastic case. The (initial) step size parameters are
\begin{itemize}
 \item PDHG, PA-PDHG, Pesquet\&Repetti: $\sigma_i^{(0)} = \tau^{(0)} = \gamma / \|\mA\| \approx 0.35$
 \item SPDHG, PA-SPDHG: $\sigma_i^{(0)} = \gamma / \|\mA_i\| \approx 0.50$, $\tau^{(0)} = \gamma / (n \max_i \|\mA_i\|) \approx 0.25$
\end{itemize}
The step sizes for acceleration vary with the iteration with the primal step size $\tk$ getting smaller and the dual step size $\sk$ getting larger. The extrapolation factor $\theta$ is chosen to be 1 for non-accelerated and converging to 1 for accelerated algorithms.

\paragraph{Results}
The quantitative results in \cref{fig:case2:L2TV:quant} show that the accelerated algorithms are much faster than the non-accelerated versions. Moreover, it can be seen that the stochastic variant of the accelerated PA-PDHG is even faster than its deterministic variant. In addition, the results show that the accelerated SPDHG indeed converges as $1/K^2$ in the norm of the primal part. Visual assessment of the denoised images in \cref{fig:case2:L2TV:images} confirms these conclusions.

\subsection{Huber-TV Deblurring (Dual Acceleration)}
In the third example we consider deblurring with known convolution kernel where the forward operator $\mA_1$ resembles the convolution of images in $\sX = \sR^{d_1 \times d_2}, d_1 = 408, d_2 = 544$ with a motion blur of size $15 \times 15$. The noise is modeled to be Poisson with a constant background of 200 compared to the approximate data mean of $694.3$. We further assume to have the knowledge that the reconstructed image should be non-negative and upper-bounded by 100. By the nature of the forward operator $\mA x \geq 0$ whenever $x \geq 0$. Therefore the solution to \cref{EQU:MINPROB} with the Kullback--Leibler divergence \cref{EQU:EXAMPLES:KL} remains the same if we replace the Kullback--Leibler divergence by the differentiable
\begin{align}
f_1(y) = \sum_{i=1}^N \begin{cases}
        y_i + r_i - b_i + b_i \log\left(\frac{b_i}{y_i + r_i}\right) & \text{if $y_i \geq 0$} \\
        \frac{b_i}{2 r_i^2} y_i^2 + \left(1 - \frac{b_i}{r_i}\right) y_i + r_i - b_i + b_i \log\left(\frac{b_i}{r_i}\right) & \text{else}
        \end{cases}
        \label{EQU:EXAMPLES:SMOOTHKL}
\end{align}
which has a $\left(\max_i b_i/r_i^2\right)$ Lipschitz continuous gradient. The Lipschitz constant is well-defined and non-zero as both the data $b_i$ as well as the background $r_i$ are positive. In our numerically example it is approximately $0.31$.

\newcommand{\drawpic}[1]{
    \node [anchor=north west] at (0,0) {\includegraphics[width=\ImWidth]{pics/deblurring_1k2/#1}};
    \node [anchor=south west] at (-0.1cm,-5.5cm) {\includegraphics[clip, trim=330px 380px 0px 70px, width=2cm]{pics/deblurring_1k2/#1}};}
\begin{figure}%
\def\ImWidth{3.8cm}%
\centering%
\includegraphics[height=\StatsHeight]{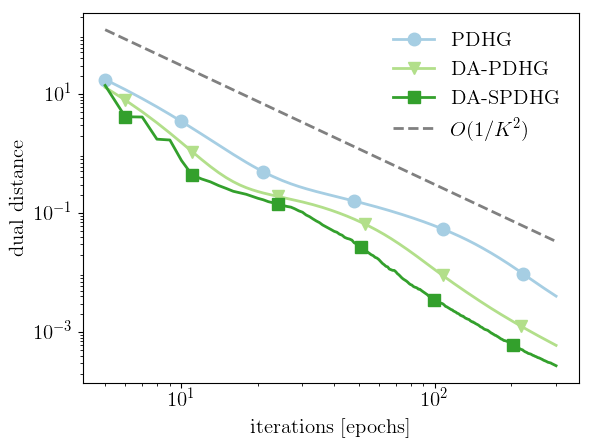} \hspace*{5mm}%
\includegraphics[height=\StatsHeight]{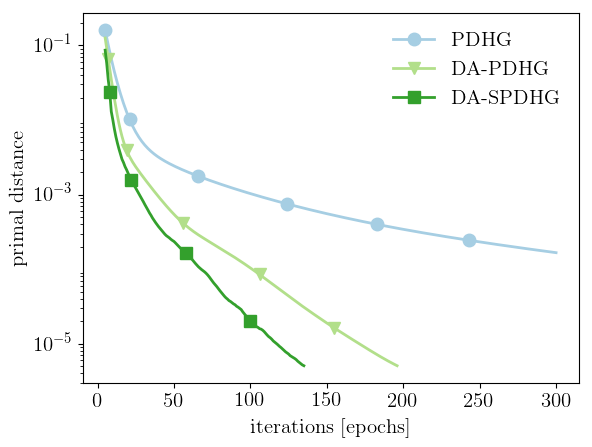}%
\caption{Dual acceleration for Huber-TV deblurring. Acceleration speeds up the convergence of both the dual variable (\textbf{left}) and the primal variable (\textbf{right}). Randomization in conjunction with acceleration yields even faster convergence. The accelerated algorithms converge with $O(1/\Iter^2)$ in the dual distance (dashed line).} \label{FIG:EXAMPLES:DEBLURRING:QUANT}%
\vspace*{4mm}%
\begin{tikzpicture}%
    \drawpic{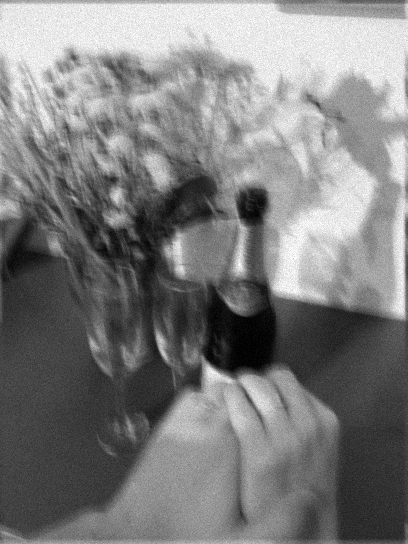}%
    \node [anchor=north west] at (0,0) {\includegraphics[width=1cm]{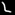}};%
\end{tikzpicture}%
\begin{tikzpicture} \drawpic{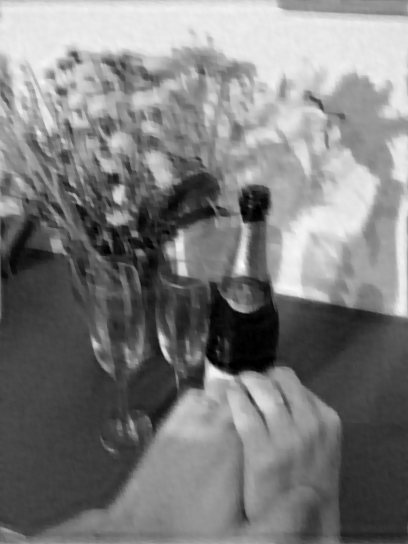} \end{tikzpicture}%
\begin{tikzpicture} \drawpic{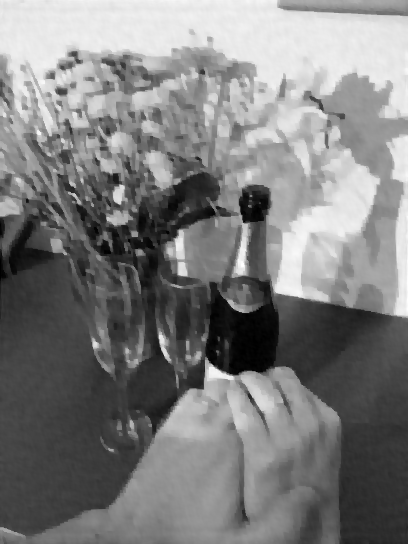} \end{tikzpicture}%
\caption{Results after 50 epochs for deblurring with Huber-TV. \textbf{From left to right:} Blurry and noisy data with kernel (magnified), PDHG and DA-SPDHG.} \label{FIG:EXAMPLES:DEBLURRING:IMAGES}%
\end{figure}

Prior smoothness information is represented by the anisotropic TV with Huberized norm
\begin{align*}
    f_i(\mA_i x) = \alpha \sum_j \begin{cases} |y_j| & \text{if $|y_j| > \eta$} \\ \frac1{2\eta} |y_j|^2 + \frac\eta2 & \text{else} \end{cases}
\end{align*}
for $i = 2, 3$ where $y_j = \nabla_{i-1} x_j$ are finite differences, $\eta = 1$ and regularization parameter $\alpha = 0.1$. The constraints on the image are enforced by the characteristic function $g = \imath_{\sB}$ with $\sB = \{x \in \sX \, | \, 0 \leq x_j \leq 100\}$.

The convex conjugate of the modified Kullback--Leibler divergence \cref{EQU:EXAMPLES:SMOOTHKL} is
\begin{align*}
f_1^\ast(z) = \sum_{i=1}^N
    \begin{cases}
    \frac{r_i^2}{2 b_i} z_i^2 + \left(r_i - \frac{r_i^2}{b_i}\right)z_i + \frac{r_i^2}{2 b_i} + \frac{3b_i}{2} - 2 r_i - b_i \log\left(\frac{b_i}{r_i}\right) & \text{if $z_i < 1 - \frac{b_i}{r_i}$}\\
    -r_i z_i - b_i \log(1-z_i) & \text{if $1 - \frac{b_i}{r_i} \leq z_i < 1$}\\
    \infty & \text{if $z_i \geq 1$}
    \end{cases}
\end{align*}
which is $\left(\min_i \frac{r_i^2}{b_i}\right)$-strongly convex with proximal operator
\begin{align*}
  [\Prox{\sigma}{f^\ast_1}{z}]_i =
 \begin{cases}
 \frac{ b_i z_i - \sigma r_i b_i + \sigma r_i^2 } {b_i + \sigma r_i^2} & \text{if $z_i < 1 - \frac{b_i}{r_i}$} \\
 \frac12 \Bigl\{ z_i + \sigma r_i + 1 - \sqrt{(z_i + \sigma r_i - 1)^2 + 4 \sigma b_i} \Bigr\} & \text{else}
 \end{cases} \, .
\end{align*}

\paragraph{Parameters} In this experiment we choose $\gamma = 0.99$ and consider uniform sampling, i.e. $p_i = 1/n$. The number of subsets are either $n=1$ in the deterministic case or $n=3$ in the stochastic case. The (initial) step size parameters are chosen to be
\begin{itemize}
 \item PDHG: $\sigma_i = \tau = \gamma /\|\mA\| \approx 0.095$
 \item DA-PDHG: $\tilde \sigma^{(0)}_i = \mu_f / \|\mA\| \approx 0.096$ , $\tau^{(0)} = \gamma /\|\mA\| \approx 0.095$
 \item DA-SPDHG: $\tilde \sigma^{(0)} = \min_i \frac{\mu_i p_i^2}{\tau^0 \|\mA_i\|^2 + 2 \mu_i p_i (1-p_i)} \approx 0.073$, \\ \phantom{DA-SPDHG: }$\tau^{(0)} = 1 / (n \max_i \|\mA_i\|) \approx 0.032$
\end{itemize}

\paragraph{Results}
The quantitative results in \cref{FIG:EXAMPLES:DEBLURRING:QUANT} show that the algorithm converges indeed with rate $O(1/\Iter^2)$ as proven in \cref{THM:SEMISTRONGLY:DUAL}. Moreover, they also show that randomization and acceleration can be used in conjunction for further speed-ups. Some example images are shown in \cref{FIG:EXAMPLES:DEBLURRING:IMAGES} which show that randomization may lead to sharper images with the same number of epochs.

\subsection{PET Reconstruction (Linear Rate)}
For the final example we turn back to PET reconstruction but this time with linear convergence rate. This means we want to solve the same minimization problem as in the first example, but now we replace the Kullback--Leibler functional by its modified version as in the previous example. We note again that this does not change the solution of the minimization problem. Moreover, to make TV strongly convex we add another regularization term $\mu / 2 \|x\|_2^2$ to $g$. Note that the proximal operator of TV (indeed any functional) with added squared $\ell^2$-norm, i.e. $g(x) = \alpha \operatorname{TV}(x) + \mu / 2 \|x\|_2^2$, can be solved by means of the original proximal operator $\Prox{\sigma}{g}{z} = \Prox{\sigma \alpha / (1 + \sigma \mu)}{\operatorname{TV}}{z / (1 + \sigma \mu)}$. The regularization parameters are chosen as $\alpha = 0.05$ and $\mu = 0.5$.

\paragraph{Parameters} In this experiment we choose $\rho = 0.99$ and the sampling to be uniform as the operators $\mA_i$ all have similar norms. The step size parameters are chosen as derived in \cref{SEC:PARAM}, in particular, we choose
\begin{itemize}
 \item PDHG: $\sigma \approx 3.8 \cdot 10^{-4}, \tau \approx 4.8 \cdot 10^{-3}, \theta \approx 0.995$
 \item Pesquet\&Repetti: $\sigma_i = \tau = \gamma /\|\mA\| \approx 1.4 \cdot 10^{-3}$
 \item SPDHG ($n=10$ subsets): $\sigma_i \approx 1.2 \cdot 10^{-3}$, $\tau \approx 1.5 \cdot 10^{-3}$, $\theta^n \approx 0.985$
 \item SPDHG ($n=50$): $\sigma_i \approx 2.4 \cdot 10^{-3}$, $\tau \approx 5.8 \cdot 10^{-4}$, $\theta^n \approx 0.971$
\end{itemize}
Note that the contraction rates of one epoch $\theta^n$ already indicate that SPDHG ($n=50$) may be faster than PDHG and SPDHG ($n=10$).

\paragraph{Results} The quantitative results in \cref{FIG:STRCONVEX:QUANT} in terms of both distance to saddle point and objective value show that randomization speeds up the convergence so that both SPDHG and the algorithm of Pesquet\&Repetti are faster than the deterministic PDHG. Interestingly, while more subsets make SPDHG faster, this does not hold for the algorithm of Pesquet\&Repetti where the speed seems to be constant with respect to the number of subsets. Moreover, the plot on the left confirms the linear convergence as proven in \cref{THM:STRONGLY}. The visual results in \cref{FIG:STRCONVEX:QUAL} confirm these observations as SPDHG with 50 subsets and 10 epochs is (in contrast to PDHG) visually already very close to the saddle point.

\begin{figure}
\def\ImWidth{3.5cm}
\centering
\includegraphics[height=\StatsHeight]{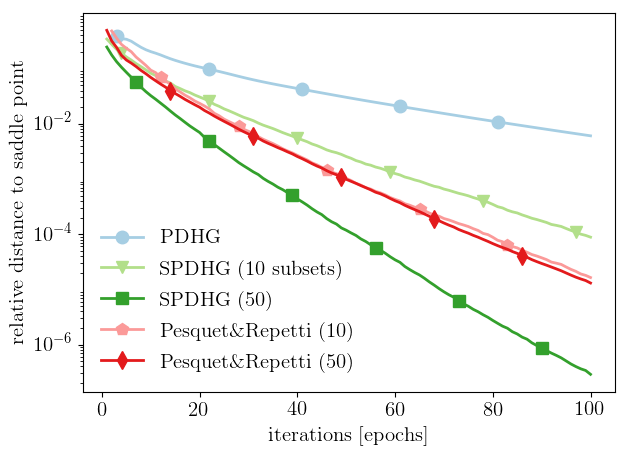} \hspace*{5mm}%
\includegraphics[height=\StatsHeight]{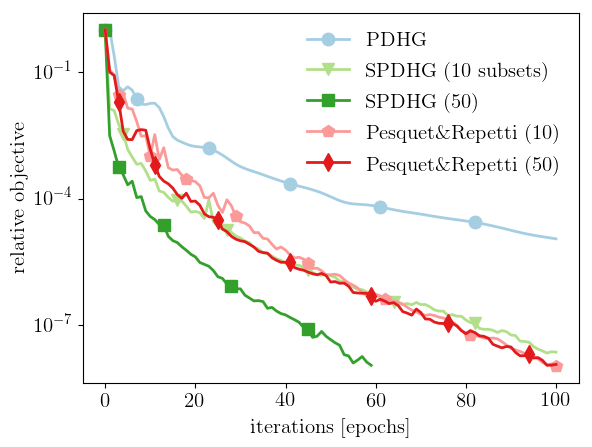}%
\caption{PET reconstruction with a strongly convex TV prior. Both the distance to the saddle point (\textbf{left}) and the objective value (\textbf{right}) show the speed-up by randomization over the deterministic PDHG. Moreover, for 50 subsets SPDHG is much faster than the algorithm proposed by Pesquet\&Repetti. Also note the linear convergence on the left as proven in \cref{THM:STRONGLY}.} \label{FIG:STRCONVEX:QUANT}%
\vspace*{4mm}
\includegraphics[clip, trim=4mm 4mm 4mm 4mm, width=\ImWidth]{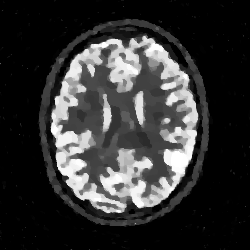}\hspace*{5mm}%
\includegraphics[clip, trim=4mm 4mm 4mm 4mm, width=\ImWidth]{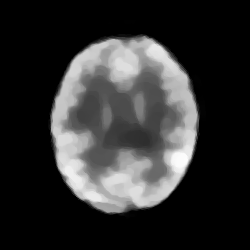}\hspace*{1mm}%
\includegraphics[clip, trim=4mm 4mm 4mm 4mm, width=\ImWidth]{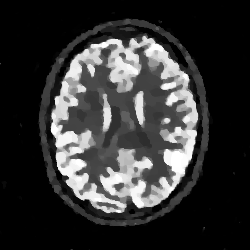}%
\caption{PET reconstruction results after 10 epochs. \textbf{Left:} Approximate primal part of saddle point computed by 2000 iterations of PDHG. \textbf{Right:} PDHG and SPDHG (50 subsets).} \label{FIG:STRCONVEX:QUAL}%
\end{figure}%

\section{Conclusions and Future Work}
We proposed a natural stochastic generalization of the deterministic PDHG algorithm to convex-concave saddle point problems that are separable in the dual variable. The analysis was carried out in the context of {\em arbitrary samplings} which enabled us to obtain known deterministic convergence results as special cases. We proposed optimal choices of the step size parameters with which the proposed algorithm showed superior empirical performance on a variety of optimization problems in imaging.

In the future, we would like to extend the analysis to include iteration dependent (adaptive) probabilities \cite{Csiba2015} and strong convexity parameters to further exploit the structure of many relevant problems. Moreover, the present optimal sampling strategies are only for scalar-valued step sizes and serial sampling. In the future, we wish to extend this to other sampling strategies such as multi-block or parallel sampling.

\appendix

\section{Postponed Proofs}

\begin{proof}[Proof of \cref{LEM:ESO}]
With the definition of $y^+$ and $\mQ$ we have by completing the norm for any $x \in \sX$ that
\begin{align}
  2\Angle{\mQ \mA x}{y^+ - y}
  &= 2\left\langle x, \sum_{i \in \sS} \mAa_i p_i^{-1}(\hat y_i - y_i) \right\rangle \notag \\
  &= 2\left\langle c^{1/2} \mT^{-1/2} x, c^{-1/2} \sum_{i \in \sS} \mC_i^\ast p_i^{-1} \mS_i^{-1/2}(\hat y_i - y_i) \right\rangle \notag \\
  &\geq - c \|x\|^2_{\mT^{-1}} - \frac{1}{c} \left\|\sum_{i \in \sS} \mC_i^\ast z_i \right\|^2 \, . \label{EQU:ESO:1}
\end{align}
where we used $z_i \De p_i^{-1} \mS_i^{-1/2}(\hat y_i - y_i)$. Moreover, the expectation of the second term of the right hand side of \cref{EQU:ESO:1} can be estimated as
\begin{align}%
  \E_{\sS} \left\|\sum_{i \in \sS} \mC_i^\ast z_i \right\|^2
  \leq \sum_{i=1}^n p_i v_i \|z_i\|^2
  \leq \left(\max_i \frac{v_i}{p_i}\right) \sum_{i=1}^n p_i^2 \|z_i\|^2 \label{EQU:ESO:2}
\end{align}
where the first inequality is due to the ESO inequality \cref{EQU:ESO}. Inserting $z$ leads to
\begin{align}%
\sum_{i=1}^n p_i^2 \|z_i\|^2
  = \sum_{i=1}^n p_i \|\hat y_i - y_i\|^2_{p_i^{-1} \mS_i^{-1}}
  = \E_{\sS} \|y^+ - y\|^2_{\mQ \mS^{-1}} \label{EQU:ESO:3}
\end{align}
where the last equation holds true by the definition of the expectation. Combining the expected value of inequality \cref{EQU:ESO:1} with \cref{EQU:ESO:2,EQU:ESO:3} yields the assertion.
\end{proof}

\begin{proof}[Proof of \cref{LEM:STD}]
By the definition of the proximal operator, for any $(x, y) \in \sW$ it holds that
\begin{align*}
    g(x) &\geq g(\xkp) + \Angle{\mTk^{-1}(\xk - \xkp) - \mAa \ybk}{x - \xkp} + \frac{\mu_g}{2} \|x - \xkp\|^2 \\
    f^\ast_i(y_i) &\geq f^\ast_i(\yhkp_i) + \Angle{[\mSk]_i^{-1}(\yk_i - \yhkp_i) + \mA_i \xkp}{y_i - \yhkp_i} + \frac{\mu_i}{2} \|y - \yhkp\|^2
\end{align*}
for $i = 1, \ldots, n$. Summing twice all inequalities and exploiting the identity
$$2 \Angle{\mB(a - b)}{c - b} = \|a - b\|^2_\mB + \|b - c\|^2_\mB - \|a - c\|^2_\mB$$
yields
\begin{align*}
\|\xk - x\|_{\mTk^{-1}}^2 + \|\yk - y\|_{\mSk^{-1}}^2
&\geq \|\xkp - x\|_{\mTk^{-1} + \mu_g \mI}^2 + \|\yhkp - y\|_{\mSk^{-1} + \mM}^2 \\
&\qquad + 2 \left(g(\xkp) - g(x) + f^\ast(\yhkp) - f^\ast(y)\right) \\
&\qquad + 2 \left(\Angle{\mA \xkp}{y - \yhkp} - \Angle{\mA(x - \xkp)}{\ybk} \right)\\
&\qquad + \|\xkp - \xk\|_{\mTk^{-1}}^2 + \|\yhkp - \yk\|_{\mSk^{-1}}^2
\end{align*}
where we used the definition of the inner product and the norm on the product space $\sY$.
It now suffices to complete the generalized distances $\cG(\xkp | w)$ and $\cF(\yhkp | w)$.
\end{proof}

\begin{proof}[Proof of \cref{LEM:STD:STOCHASTIC}]
We follow a similar line of arguments as in the proof of \cref{THE:NONSTRCVX}. Note that for any saddle point $\wo = (\xo, \yo)$ we have $$2 \cH(\xkp, \yhkp | \wo) = 2 D_h^q(\xkp, \yhkp, \wo) \geq \|\xkp - \xo\|^2_{\mu_g} + \|\yhkp - \yo\|^2_{\mM}$$
such that the estimate of \cref{LEM:STD} can be written with $w = \wo$ as
\begin{align*}
\|\xk - \xo\|_{\mTk^{-1}}^2 + \|\yk - \yo\|^2_{\mSk^{-1}}
&\geq \|\xkp - \xo\|_{\mTk^{-1} + 2 \mu_g\mI}^2 + \|\yhkp - \yo\|^2_{\mSk^{-1} + 2 \mM} \\
&\qquad - 2 \langle \mA(\xkp - \xo), \yhkp - \ybk \rangle \\
&\qquad + \|\xkp - \xk\|_{\mTk^{-1}}^2 + \|\yhkp - \yk\|^2_{\mSk^{-1}}
\end{align*}
using the rule \cref{EQU:MATRIXID}. With \cref{EQU:THM1:N,EQU:THM1:V} and again \cref{EQU:MATRIXID} we arrive at
\begin{align}
& \quad \; \|\xk - \xo\|_{\mTk^{-1}}^2 + \|\yk - \yo\|^2_{\mQ \mSk^{-1} + 2 \mM(\mQ - \mI)} \notag \\
&\geq \Ekp \biggl\{ \|\xkp - \xo\|_{\mTk^{-1} + 2 \mu_g\mI}^2 + \|\yhkp - \yo\|^2_{\mQ\mSk^{-1} + 2 \mM\mQ} \notag\\
&\qquad - 2 \langle \mA(\xkp - \xo), \mQ (\ykp - \yk) + \yk - \ybk \rangle \notag\\
&\qquad + \|\xkp - \xk\|_{\mTk^{-1}}^2 + \|\ykp - \yk\|^2_{\mQ \mSk^{-1}} \biggr\} \, . \label{EQU:LEM:STDSTOCH:1}
\end{align}
Inserting the extrapolation \cref{EQU:LEM:EXTRA} into the inner product yields
\begin{align}
& \quad \; - \thkm \langle \mQ \mA(\xk - \xo), \yk - \ykm \rangle \notag\\
&\geq \Ekp \biggl\{ \thkm \langle \mQ \mA(\xkp - \xk), \yk - \ykm \rangle \notag\\
&\qquad - \langle \mQ \mA(\xkp - \xo), \ykp - \yk \rangle \biggr\} \, . \label{EQU:LEM:STDSTOCH:IP}
\end{align}
The assertion is shown by taking the expectations $\E^{(\iter, \iter-1)} \De \Ek\Ekm$ on \cref{EQU:LEM:STDSTOCH:1}, using \cref{EQU:LEM:STDSTOCH:IP} and estimating the last inner product by \cref{LEM:ESO} as
\begin{align*}
&\quad \; 2 \thkm \Ekp \Angle{\mQ \mA (\xkp - \xk)}{\yk - \ykm} \\
&\geq - \Ekp \left\{ \|\xkp - \xk\|_{\mTk^{-1}}^2 + (\gamma \thkm)^2 \|\yk - \ykm\|^2_{\mQ \mSk^{-1}} \right\} \, .
\end{align*}
\end{proof}

\bibliographystyle{siamplain}
\bibliography{library}

\begin{thebibliography}{10}

\bibitem{Adler2017}
{\sc J.~Adler, H.~Kohr, and O.~\"Oktem}, {\em {Operator Discretization Library
  (ODL)}}, 2017, \url{https://github.com/odlgroup/odl}.

\bibitem{Allen-Zhu2016}
{\sc Z.~Allen-Zhu, Y.~Yuan, P.~Richt{\'{a}}rik, and Y.~Yuan}, {\em {Even Faster
  Accelerated Coordinate Descent Using Non-Uniform Sampling}}, International
  Conference on Machine Learning, 48 (2016),
  \url{https://arxiv.org/abs/1512.09103}.

\bibitem{Balamurugan2016a}
{\sc P.~Balamurugan and F.~Bach}, {\em {Stochastic Variance Reduction Methods
  for Saddle-Point Problems}},  (2016), pp.~1--23,
  \url{https://arxiv.org/abs/1605.06398}.

\bibitem{Bauschke2011}
{\sc H.~H. Bauschke and P.~L. Combettes}, {\em {Convex Analysis and Monotone
  Operator Theory in Hilbert Spaces}}, 2011,
  \url{https://doi.org/10.1007/978-1-4419-9467-7}.

\bibitem{Beck2009a}
{\sc A.~Beck and M.~Teboulle}, {\em {A Fast Iterative Shrinkage-Thresholding
  Algorithm for Linear Inverse Problems}}, SIAM Journal on Imaging Sciences, 2
  (2009), pp.~183--202, \url{https://doi.org/10.1137/080716542}.

\bibitem{Benning2016}
{\sc M.~Benning, C.-B. Sch{\"{o}}nlieb, T.~Valkonen, and
  V.~Vla{\v{c}}i{\'{c}}}, {\em {Explorations on Anisotropic Regularisation of
  Dynamic Inverse Problems by Bilevel Optimisation}}.
\newblock 2016, \url{https://arxiv.org/abs/1602.01278}.

\bibitem{Bertsekas2011a}
{\sc D.~P. Bertsekas}, {\em {Incremental Gradient, Subgradient, and Proximal
  Methods for Convex Optimization: A Survey}}, in Optimization for Machine
  Learning, S.~{Sra, S. and Nowozin, S. and Wright}, ed., MIT Press, 2011,
  pp.~85--120.

\bibitem{Bertsekas2011}
{\sc D.~P. Bertsekas}, {\em {Incremental Proximal Methods for Large Scale
  Convex Optimization}}, Mathematical Programming, 129 (2011), pp.~163--195,
  \url{https://doi.org/10.1007/s10107-011-0472-0}.

\bibitem{Blatt2007}
{\sc D.~Blatt, A.~O. Hero, and H.~Gauchman}, {\em {A Convergent Incremental
  Gradient Method with a Constant Step Size}}, SIAM Journal on Optimization, 18
  (2007), pp.~29--51, \url{https://doi.org/10.1137/040615961}.

\bibitem{Bredies2015}
{\sc K.~Bredies and M.~Holler}, {\em {A TGV-Based Framework for Variational
  Image Decompression, Zooming, and Reconstruction. Part I: Analytics}}, SIAM
  Journal on Imaging Sciences, 8 (2015), pp.~2814--2850,
  \url{https://doi.org/10.1137/15M1023865}.

\bibitem{Chambolle2004}
{\sc A.~Chambolle}, {\em {An Algorithm for Total Variation Minimization and
  Applications}}, Journal of Mathematical Imaging and Vision, 20 (2004),
  pp.~89--97.

\bibitem{Chambolle2011}
{\sc A.~Chambolle and T.~Pock}, {\em {A First-Order Primal-Dual Algorithm for
  Convex Problems with Applications to Imaging}}, Journal of Mathematical
  Imaging and Vision, 40 (2011), pp.~120--145,
  \url{https://doi.org/10.1007/s10851-010-0251-1}.

\bibitem{Chambolle2016}
{\sc A.~Chambolle and T.~Pock}, {\em {An Introduction to Continuous
  Optimization for Imaging}}, Acta Numerica, 25 (2016), pp.~161--319,
  \url{https://doi.org/10.1017/S096249291600009X}.

\bibitem{Chambolle2016a}
{\sc A.~Chambolle and T.~Pock}, {\em {On the Ergodic Convergence Rates of a
  First-Order Primal-Dual Algorithm}}, vol.~159, Springer Berlin Heidelberg,
  2016, \url{https://doi.org/10.1007/s10107-015-0957-3}.

\bibitem{Combettes2015}
{\sc P.~L. Combettes and J.-C. Pesquet}, {\em {Stochastic Quasi-Fej{\'{e}}r
  Block-Coordinate Fixed Point Iterations with Random Sweeping}}, SIAM Journal
  on Optimization, 25 (2015), pp.~1221--1248,
  \url{https://doi.org/10.1137/140971233}.

\bibitem{Csiba2015}
{\sc D.~Csiba, Z.~Qu, and P.~Richt\'{a}rik}, {\em {Stochastic Dual Coordinate
  Ascent with Adaptive Probabilities}}, Proceedings of The 32nd International
  Conference on Machine Learning, 37 (2015), pp.~674--683.

\bibitem{Dang2014}
{\sc C.~D. Dang and G.~Lan}, {\em {Randomized Methods for Saddle Point
  Computation}},  (2014), pp.~1--29, \url{https://arxiv.org/abs/1409.8625}.

\bibitem{Oliviera2016}
{\sc R.~M. de~Oliviera, E.~S. Helou, and E.~F. Costa}, {\em {String-Averaging
  Incremental Subgradients for Constrained Convex Optimization with
  Applications to Reconstruction of Tomographic Images}}, Inverse Problems, 32
  (2016), p.~115014, \url{https://doi.org/10.1088/0266-5611/32/11/115014}.

\bibitem{Defazio2014}
{\sc A.~Defazio, F.~Bach, and S.~Lacoste-Julien}, {\em {SAGA: A Fast
  Incremental Gradient Method With Support for Non-Strongly Convex Composite
  Objectives}}, Nips,  (2014), pp.~1--12,
  \url{https://arxiv.org/abs/arXiv:1407.0202v2}.

\bibitem{Esser2010}
{\sc E.~Esser, X.~Zhang, and T.~F. Chan}, {\em {A General Framework for a Class
  of First Order Primal-Dual Algorithms for Convex Optimization in Imaging
  Science}}, SIAM Journal on Imaging Sciences, 3 (2010), pp.~1015--1046,
  \url{https://doi.org/10.1137/09076934X}.

\bibitem{Estellers2015}
{\sc V.~Estellers, S.~Soatto, and X.~Bresson}, {\em {Adaptive Regularization
  With the Structure Tensor}}, IEEE Transactions on Image Processing, 24
  (2015), pp.~1777--1790, \url{https://doi.org/10.1109/TIP.2015.2409562}.

\bibitem{Fercoq2015}
{\sc O.~Fercoq and P.~Bianchi}, {\em {A Coordinate Descent Primal-Dual
  Algorithm with Large Step Size and Possibly Non Separable Functions}}.
\newblock 2015, \url{https://arxiv.org/abs/1508.04625}.

\bibitem{APPROX}
{\sc O.~Fercoq and P.~Richt\'{a}rik}, {\em {Accelerated, Parallel and PROXimal
  Coordinate Descent}}, SIAM Journal on Optimization, 25 (2015),
  pp.~1997--2023.

\bibitem{Gao2016}
{\sc X.~Gao, Y.~Xu, and S.~Zhang}, {\em {Randomized Primal-Dual Proximal Block
  Coordinate Updates}}, arXiv preprint arXiv:1605.05969,  (2016),
  \url{https://arxiv.org/abs/1605.05969}.

\bibitem{Gilboa2015}
{\sc G.~Gilboa, M.~Moeller, and M.~Burger}, {\em {Nonlinear Spectral Analysis
  via One-homogeneous Functionals - Overview and Future Prospects}}.
\newblock 2015, \url{https://arxiv.org/abs/1510.01077}.

\bibitem{Knoll2016}
{\sc F.~Knoll, M.~Holler, T.~Koesters, R.~Otazo, K.~Bredies, and D.~K.
  Sodickson}, {\em {Joint MR-PET Reconstruction using a Multi-Channel Image
  Regularizer}}, IEEE Transactions on Medical Imaging, 0062,
  \url{https://doi.org/10.1109/TMI.2016.2564989}.

\bibitem{mS2GD}
{\sc Kone\v{c}n\'{y}, J.~Liu, P.~Richt\'{a}rik, and M.~Tak\'{a}\v{c}}, {\em
  {Mini-Batch Semi-Stochastic Gradient Descent in the Proximal Setting}}, IEEE
  Journal of Selected Topics in Signal Processing, 10 (2016), pp.~242--255.

\bibitem{Kongskov2017}
{\sc R.~D. Kongskov, Y.~Dong, and K.~Knudsen}, {\em {Directional Total
  Generalized Variation Regularization}}, 2 (2017), pp.~1--24,
  \url{https://arxiv.org/abs/1701.02675}.

\bibitem{Lions1979}
{\sc P.-L. Lions and B.~Mercier}, {\em {Splitting Algorithms for the Sum of Two
  Nonlinear Operators}}, SIAM Journal on Numerical Analysis, 16 (1979),
  pp.~964--979.

\bibitem{Nedic2001}
{\sc A.~Nedi{\'{c}} and D.~P. Bertsekas}, {\em {Incremental Subgradient Methods
  for Nondifferentiable Optimization}}, SIAM J. Optimization, 12 (2001),
  pp.~109--138, \url{https://doi.org/10.1109/CDC.1999.832908}.

\bibitem{Ollinger1997a}
{\sc J.~M. Ollinger and J.~A. Fessler}, {\em {Positron Emisson Tomography}},
  IEEE Signal Processing Magazine, 14 (1997), pp.~43--55,
  \url{https://doi.org/10.1109/79.560323}.

\bibitem{Parikh2014}
{\sc N.~Parikh and S.~P. Boyd}, {\em {Proximal Algorithms}}, Foundations and
  Trends in Optimization, 1 (2014), pp.~123--231,
  \url{https://doi.org/10.1561/2400000003}.

\bibitem{Peng2016}
{\sc Z.~Peng, T.~Wu, Y.~Xu, M.~Yan, and W.~Yin}, {\em {Coordinate Friendly
  Structures, Algorithms and Applications}}, Annals of Mathematical Sciences
  and Applications, 1 (2016), pp.~1--54,
  \url{https://doi.org/10.4310/AMSA.2016.v1.n1.a2}.

\bibitem{Pesquet2015}
{\sc J.-C. Pesquet and A.~Repetti}, {\em {A Class of Randomized Primal-Dual
  Algorithms for Distributed Optimization}},  (2015),
  \url{https://arxiv.org/abs/1406.6404}.

\bibitem{Pock2011}
{\sc T.~Pock and A.~Chambolle}, {\em {Diagonal Preconditioning for First Order
  Primal-Dual Algorithms in Convex Optimization}}, Proceedings of the IEEE
  International Conference on Computer Vision,  (2011), pp.~1762--1769,
  \url{https://doi.org/10.1109/ICCV.2011.6126441}.

\bibitem{Pock2009}
{\sc T.~Pock, D.~Cremers, H.~Bischof, and A.~Chambolle}, {\em {An algorithm for
  minimizing the Mumford-Shah functional}}, Proceedings of the IEEE
  International Conference on Computer Vision,  (2009), pp.~1133--1140,
  \url{https://doi.org/10.1109/ICCV.2009.5459348}.

\bibitem{Qu2016}
{\sc Z.~Qu and P.~Richt{\'{a}}rik}, {\em {Coordinate Descent with Arbitrary
  Sampling I: Algorithms and Complexity}}, Optimization Methods and Software,
  (2014), p.~32, \url{https://doi.org/10.1080/10556788.2016.1190360}.

\bibitem{Qu2015}
{\sc Z.~Qu and P.~Richt{\'{a}}rik}, {\em {Quartz: Randomized Dual Coordinate
  Ascent with Arbitrary Sampling}}, Neural Information Processing Systems,
  (2015), pp.~1--34.

\bibitem{Richtarik2014}
{\sc P.~Richt{\'{a}}rik and M.~Tak{\'{a}}{\v{c}}}, {\em {Iteration Complexity
  of Randomized Block-Coordinate Descent Methods for Minimizing a Composite
  Function}}, Mathematical Programming, 144 (2014), pp.~1--38,
  \url{https://doi.org/10.1007/s10107-012-0614-z}.

\bibitem{Richtarik2016}
{\sc P.~Richt{\'{a}}rik and M.~Tak{\'{a}}{\v{c}}}, {\em {On Optimal
  Probabilities in Stochastic Coordinate Descent Methods}}, Optimization
  Letters, 10 (2016), pp.~1233--1243,
  \url{https://doi.org/10.1007/s11590-015-0916-1}.

\bibitem{PCDM}
{\sc P.~Richt{\'{a}}rik and M.~Tak{\'{a}}{\v{c}}}, {\em Parallel coordinate
  descent methods for big data optimization}, Mathematical Programming, 156
  (2016), pp.~433--484.

\bibitem{Rigie2015}
{\sc D.~Rigie and P.~{La Riviere}}, {\em {Joint Reconstruction of
  Multi-Channel, Spectral CT Data via Constrained Total Nuclear Variation
  Minimization}}, Physics in Medicine and Biology, 60 (2015), pp.~1741--1762,
  \url{https://doi.org/10.1088/0031-9155/60/4/1741}.

\bibitem{Rosasco2015}
{\sc L.~Rosasco and S.~Villa}, {\em {Stochastic Inertial Primal-Dual
  Algorithms}},  (2015), pp.~1--15,
  \url{https://arxiv.org/abs/arXiv:1507.00852v1}.

\bibitem{Schmidt2016}
{\sc M.~Schmidt, N.~{Le Roux}, and F.~Bach}, {\em {Minimizing Finite Sums with
  the Stochastic Average Gradient}}, Mathematical Programming,  (2016),
  pp.~1--30, \url{https://doi.org/10.1007/s10107-016-1030-6}.

\bibitem{MinibatchSGD}
{\sc M.~Takáč, A.~Bijral, P.~Richt\'{a}rik, and N.~Srebro}, {\em {Mini-Batch
  Primal and Dual Methods for SVMs}}, in In Proceedings of the 30th
  International Conference on Machine Learning, 2013.

\bibitem{Tseng1998}
{\sc P.~Tseng}, {\em {An Incremental Gradient(-Projection) Method with Momentum
  Term and Adaptive Stepsize Rule}}, SIAM Journal on Optimization, 8 (1998),
  pp.~506--531.

\bibitem{Valkonen2016}
{\sc T.~Valkonen}, {\em {Block-Proximal Methods with Spatially Adapted
  Acceleration}},  (2016), \url{https://arxiv.org/abs/1609.07373}.

\bibitem{VanAarle2016}
{\sc W.~van Aarle, W.~J. Palenstijn, J.~Cant, E.~Janssens, F.~Bleichrodt,
  A.~Dabravolski, J.~{De Beenhouwer}, K.~{Joost Batenburg}, and J.~Sijbers},
  {\em {Fast and Flexible X-ray Tomography using the ASTRA Toolbox}}, Optics
  Express, 24 (2016), p.~25129, \url{https://doi.org/10.1364/OE.24.025129}.

\bibitem{VanAarle2015}
{\sc W.~van Aarle, W.~J. Palenstijn, J.~{De Beenhouwer}, T.~Altantzis, S.~Bals,
  K.~J. Batenburg, and J.~Sijbers}, {\em {The ASTRA Toolbox: A Platform for
  Advanced Algorithm Development in Electron Tomography}}, Ultramicroscopy, 157
  (2015), pp.~35--47, \url{https://doi.org/10.1016/j.ultramic.2015.05.002},
  \url{http://dx.doi.org/10.1016/j.ultramic.2015.05.002}.

\bibitem{Wen2016}
{\sc M.~Wen, S.~Yue, Y.~Tan, and J.~Peng}, {\em {A Randomized Inertial
  Primal-Dual Fixed Point Algorithm for Monotone Inclusions}},  (2016),
  pp.~1--26, \url{https://arxiv.org/abs/1611.05142}.

\bibitem{Zhang2015}
{\sc Y.~Zhang and L.~Xiao}, {\em {Stochastic Primal-Dual Coordinate Method for
  Regularized Empirical Risk Minimization}}, Proceedings of the 32nd
  International Conference on Machine Learning,  (2015), pp.~1--34.

\bibitem{Zhong2014}
{\sc L.~W. Zhong and J.~T. Kwok}, {\em {Fast Stochastic Alternating Direction
  Method of Multipliers}}, Journal of Machine Learning Research, 32 (2014),
  pp.~46--54.

\bibitem{Zhu2015}
{\sc Z.~Zhu and A.~J. Storkey}, {\em {Adaptive Stochastic Primal-Dual
  Coordinate Descent for Separable Saddle Point Problems}}, in Machine Learning
  and Knowledge Discovery in Databases, A.~Appice, P.~P. Rodrigues, V.~S.
  Costa, C.~Soares, J.~Gama, and A.~Jorge, eds., Porto, 2015, Springer,
  pp.~643--657.

\end{thebibliography}
\end{document}